%% file: atoroidal.tex
\documentclass[12pt,fleqn]{article}

\usepackage{mathptmx}
\usepackage[T1]{fontenc}

\usepackage{latexsym,epsfig,verbatim}
\usepackage{amsmath,amsthm,amssymb}
\usepackage{bm} 
\usepackage{bbm}
\usepackage{dsfont}

\usepackage{url}
\usepackage{color}
\usepackage[colorlinks,citecolor=blue]{hyperref}

\usepackage[symbol*,stable]{footmisc}
\DefineFNsymbols{autumn}[math]{
\circ 				{\dagger} 
\bullet 			{\bullet\bullet} 			
\dagger			{\dagger\dagger}
\ddagger			{\ddagger\ddagger}					
}%
\setfnsymbol{autumn}
\renewcommand\footnotemark{}

\usepackage{accents}

\usepackage{tikz}
\usetikzlibrary{matrix,arrows,decorations.pathmorphing}
\usetikzlibrary{cd}

\numberwithin{equation}{section} 

\theoremstyle{plain}
\newtheorem*{theo}{Theorem}
\newtheorem{theorem}{Theorem}
\newtheorem{lemma}[theorem]{Lemma}

\newtheorem{corollary}[theorem]{Corollary}

\newtheorem{proposition}[theorem]{Proposition}

\theoremstyle{definition}
\newtheorem*{remark}{Remark}
\newtheorem*{notremark}{Remark on notation}

\newtheorem{question}{Question}

\input{commands.tex}

\newcommand{\colvec}[2][.7]{%
  \scalebox{#1}{%
    \renewcommand{\arraystretch}{.8}%
    $\begin{bmatrix}#2\end{bmatrix}$%
  }
}

\newcommand{\LB}{[ \hspace{-1.5pt} [ }
\newcommand{\RB}{] \hspace{-1.5pt} ] }

\begin{document}

\newcounter{ccomments}
\newcommand{\Chris}[1]{\textbf{\color{red}(C\arabic{ccomments})} \marginpar{\scriptsize\raggedright\textbf{\color{red}(C\arabic{ccomments})Chris: }#1}
	\addtocounter{ccomments}{1}}
	
\newcounter{acomments}
\newcommand{\autumn}[1]{\textbf{\color{red}(A\arabic{acomments})} \marginpar{\scriptsize\raggedright\textbf{\color{red}(A\arabic{acomments})Autumn: }#1}
	\addtocounter{acomments}{1}}

\title{\textbf{Atoroidal surface bundles}}
\author{Autumn E. Kent and Christopher J. Leininger\thanks{A. E. Kent was supported in part by the NSF grants DMS-1104871, DMS-1350075, DMS-1904130, and DMS-2202718. 
C. J. Leininger was supported in part by the NSF grants DMS-2305286, DMS-2106419, DMS-1811518, DMS-1510034, DMS-1207183, DMS-0603881, and DMS-0905748.}}

\maketitle

\begin{abstract}
We show that there is a type--preserving homomorphism from the fundamental group of the figure--eight knot complement to the mapping class group of the thrice--punctured torus. As a corollary, we obtain infinitely many commensurability classes of purely pseudo-Anosov surface subgroups of mapping class groups of closed surfaces. This gives the first examples of compact aspherical atoroidal surface bundles over surfaces.
\end{abstract}


\section{Introduction}

Mapping class groups of surfaces are often viewed in analogy with Kleinian groups, with the Teichm\"uller space playing the role of $\HH^3$ and the mapping class group playing the role of a finite covolume lattice in $\PSL_2(\CC)$. 
Finite covolume Kleinian groups contain swarms of quasifuchsian closed surface subgroups. 
In the noncocompact case, the existence of such subgroups is due to J. Masters and X. Zhang \cite{MastersZhang2008} and M. Baker and D. Cooper \cite{Baker_Cooper_2000}, while their ubiquity was proved by Cooper and D. Futer \cite{CooperFuter} and J. Kahn and A. Wright \cite{KahnWright2021}.  In the cocompact case, existence and ubiquity is due to Kahn and V. Markovi\'c \cite{KahnMarkovic2012}.
Such quasifuchsian subgroups contain no parabolic elements, and so give examples of purely hyperbolic surface subgroups in every finite covolume Kleinian group.

At the turn of the century, B. Farb and L. Mosher \cite{FarbMosher} developed a theory of convex cocompact subgroups of mapping class groups that has been refined and explored over the years---see 
\cite{FarbMosher2}, 
\cite{KentLeininger2008}, 
\cite{Hamenstaedt},
\cite{KentLeininger2008.convergence}, 
\cite{KentLeiningerSchleimer2009},
\cite{KentLeininger2007.survey},
\cite{KoberdaMangahasTaylor},
\cite{DowdallKentLeininger2014},
\cite{DurhamTaylor2015},
\cite{Fujiwara2015},
\cite{MangahasTaylor2016},
\cite{BestvinaBrombergKentLeininger2020}, 
\cite{LeiningerRussell}, and
\cite{Tshishiku2024}.

There are many equivalent definitions of convex cocompactness, see \cite{KentLeininger2008}, \cite{KentLeininger2008.convergence}, and \cite{BestvinaBrombergKentLeininger2020} for a few, but the simplest to state is that sending the subgroup of $\Mod(S)$ to its orbit in the curve complex of $S$ is a quasiisometric embedding. This was proven equivalent to convex cocompactness by the authors \cite{KentLeininger2008} and, independently, by U. Hamenst\"adt \cite{Hamenstaedt}.

In \cite{FarbMosher}, Farb and Mosher constructed convex cocompact free groups and asked whether or not there are examples of convex cocompact subgroups of $\Mod(S)$ that are not virtually free. 
\begin{question}[Questions 1.7 and 1.9 of \cite{FarbMosher}]\label{question.big}
	Are there convex cocompact subgroups of $\Mod(S)$ that are not virtually free? Are there infinite convex cocompact subgroups that are isomorphic to the fundamental group of a closed surface? Are the surface group extensions of such subgroups $\delta$--hyperbolic?
\end{question}

Any subgroup $G$ of $\Mod(S)$ gives rise to a surface group extension
\[
	1 \to \pi_1(S) \to \Gamma_G \to G \to 1.
\]
This may be constructed as follows.  Fix a basepoint $z$ in $S$, and let $S_z$ denote the surface obtained by puncturing $S$ at $z$.  Then $\Gamma_G$ is obtained by pulling $G$ back to $\Mod(S_x)$ in the Birman exact sequence \cite{Birman}
\[
	1 \to \pi_1(S,x) \to \Mod(S_x) \to \Mod(S) \to 1.
\]

Combining the work of Farb, Mosher, and Hamenst\"adt \cite{FarbMosher,Hamenstaedt} produces the following theorem characterizing the hyperbolic extensions.  See also \cite{MjSardar}.
\begin{theo}[Farb--Mosher \cite{FarbMosher}, Hamenst\"adt \cite{Hamenstaedt}]
	Suppose that $S$ is closed. Then the extension $\Gamma_G$ is $\delta$--hyperbolic if and only if $G$ is a convex cocompact subgroup of $\Mod(S)$. \qed
\end{theo}

When $G$ is a cyclic subgroup generated by a pseudo-Anosov mapping class $g$, the extension $\Gamma_G$ is the fundamental group of a hyperbolic $3$--manifold, by Thurston's Hyperbolization theorem for fibered $3$--manifolds \cite{ThurstonFibered,OtalFibered}.
It is then natural to wonder if there are closed surface bundles over surfaces that admit a hyperbolic structure as well.

\begin{question}\label{question.hyperbolic.bundle} Is there a compact surface bundle over a surface that admits a hyperbolic metric? 
\end{question}

Conjecturally, the answer to the second question is no, as it would violate certain conjectures about the Seiberg--Witten invariants of hyperbolic manifolds---see \cite{Reid}.  However, a negative answer would not forbid surface bundles over surfaces from having $\delta$--hyperbolic fundamental groups, and so we focus our attention on Question \ref{question.big}.\footnote{It \textit{is} known that such a bundle cannot be a complex hyperbolic surface \cite{Kapovich1998}.}
Note that convex cocompactness implies that the subgroup is virtually purely pseudo-Anosov \cite{FarbMosher}, and so, implicit in Question \ref{question.big} is the following question.  

\begin{question}\label{question.PPA}
	Is there a purely pseudo-Anosov closed surface group in some mapping class group $\Mod(S)$?
\end{question}

We give the first examples of such groups.

\newcommand{\TheoremInfiniteSurface}{
For any closed surface $S$ of genus $g\geq 4$, there are infinitely many commensurability classes of purely pseudo-Anosov surface subgroups of $\Mod(S)$.
}
\begin{theorem}[Purely pseudo-Anosov surface groups]\label{theorem.ppAsurfaces} \TheoremInfiniteSurface
\end{theorem}

We obtain Theorem \ref{theorem.ppAsurfaces} by constructing a nice representation of the figure--eight knot group into the mapping class group of a thrice--punctured torus and using a branched covering trick.

Let $M_8$ be the complement of the figure--eight knot.
This is a finite--volume arithmetic  hyperbolic $3$--manifold whose associated Kleinian group contains infinitely many commensurability classes of cocompact fuchsian subgroups---see \cite{Reid.1991,RIley.1975a,Riley.1975b,Maclachlan}.

\begin{theorem}[Type preserving figure eight] \label{T:type preserving fig 8} There is a type--preserving representation of the fundamental group $\pi_1(M_8)$ of the figure--eight knot complement to the mapping class group of the thrice--punctured torus.
\end{theorem}
\noindent
Here a \textit{type--preserving} representation $\Delta$ of $\pi_1(M_8)$ into a mapping class group is one for which $\Delta(\gamma)$ is pseudo-Anosov if and only if $\gamma$ is hyperbolic.  Equivalently, $\Delta(\gamma)$ is reducible if and only if $\gamma$ is parabolic (that is, $\gamma$ is peripheral in $M_8$).  See also the remark at the end of Section~\ref{Section.forgetting.punctures}. 

The abundance of purely hyperbolic surface subgroups of $\pi_1(M_8)$ and a \linebreak branched covering trick produces Theorem \ref{theorem.ppAsurfaces}. 

We say that a closed manifold is \textit{atoroidal} if its fundamental group contains no $\ZZ^2$.
Purely pseudo-Anosov surface subgroups of mapping class groups of mapping class groups of closed surfaces give rise to atoroidal surface bundles over surfaces, and we have the following theorem.

\begin{theorem}[Atoroidal surface bundles] \label{T:atoroidal surface bundles}
Let $S$ be a closed surface of genus $g \geq 4$. There are infinitely many closed aspherical atoroidal $4$--manifolds $E$ that fiber as an $S$--bundle over a surface $B$.
\end{theorem}

\noindent
J.-F. Lafont, N. Miller, and L. Ruffoni \cite{LafontMillerRuffoni} have shown that many of the bundles in Theorem~\ref{T:atoroidal surface bundles} have signature $0$, a necessary condition for real hyperbolicity.
On the other hand, the authors have recently shown \cite{KentLeininger2025} that many of these bundles do {\em not} admit a hyperbolic structure.  
Whether {\em any} of the bundles from Theorem \ref{T:atoroidal surface bundles} could provide an affirmative answer to Question~\ref{question.hyperbolic.bundle} remains an open question.

We expect that our surface subgroups from Theorem~\ref{theorem.ppAsurfaces} are convex cocompact, and plan to take up that topic in a subsequent paper.

In light of Theorem \ref{T:type preserving fig 8}, it is natural to ask which Kleinian groups admit type--preserving representations to mapping class groups of surfaces.  

\begin{question} \label{Q:hyperbolic-n} Which hyperbolic $n$--manifold groups admit type--preserving representations into mapping class groups? 
\end{question}

One of the original motivations for the study of convex cocompactness in mapping class groups was to approach the question of M. Gromov as to whether groups $\calG$ with finite $K(\calG,1)$s are hyperbolic if and only if they contain no Baum-slag--Solitar groups.  Indeed, the extension $\Gamma_G$ is hyperbolic if and only if $G$ is convex cocompact, while it has no Baumslag--Solitar subgroups if and only if $G$ is purely pseudo-Anosov. 
G. Italiano, B. Martelli, and M. Migliorini have recently shown \cite{ItalianoMartelliMigliorini}, quite remarkably, that there are in fact finite--type nonhyperbolic groups with no Baumslag--Solitar subgroups, but whether or not there are counterexamples of the form $\Gamma_G$ remains open.  At any rate, we have the following corollary of our work here.

\begin{corollary}
	There is either a $\delta$--hyperbolic surface-by-surface group or a nonhyperbolic surface-by-surface group that has no Baumslag--Solitar subgroups.
\end{corollary}
\begin{proof} See \cite{KentLeininger2007.survey}. 
\end{proof}

In the Kleinian setting, it is a theorem of Thurston \cite[Theorem 5.2.18]{CanaryEpsteinGreen} that, given a finitely generated Kleinian group $\Gamma$ and a number $\chi < 0$, there are only finitely many conjugacy classes of quasifuchsian surface subgroups of $\Gamma$ whose Euler characteristic is at least $\chi$.
B. Bowditch proved \cite{Bowditch} the analogous statement for purely pseudo-Anosov surface subgroups of mapping class groups.

\begin{theorem}[Bowditch \cite{Bowditch}] 
	Given $S$ and a number $\chi < 0$, there are only finitely many conjugacy classes of purely pseudo-Anosov subgroups of $\Mod(S)$ isomorphic to $\pi_1(\Sigma)$, where $\Sigma$ is a closed surface of Euler characteristic at least $\chi$.
\end{theorem}
\noindent Our main theorem provides a lower bound on the number of these conjugacy classes.
\newcommand{\TheoremCounting}{Let $S$ be a closed surface of genus at least $4$ or a thrice--punctured torus.  Then the number of commensurability classes of purely pseudo-Anosov subgroups of $\Mod(S)$ that are isomorphic to the fundamental group of a surface of genus at most $h$ is bounded below by a strictly increasing linear function of $h$.}
\begin{theorem}\label{theorem.counting}
\TheoremCounting
\end{theorem}
This estimate is in fact a \textit{lot} smaller than the number of all purely pseudo-Anosov surface subgroups of $\Mod(S)$ of genus at most $h$.  Indeed, after the initial circulation of this article, X. Han, Z. Rao, and J. Wan \cite{HanRaoWan} proved that the number of commensurability classes of genus $h$ purely pseudo-Anosov subgroups of $\Mod(S_g)$ is at least $(ch)^{2h}$ for some $c > 0$ once $h$ is sufficiently large.  
Their proof was inspired by that of Kahn and Markovi\'c \cite{KahnMarkovic2012Count}, who show that there is a constant $c$ depending only on the injectivity radius so that the number of commensurability classes of closed incompressible surfaces of genus $h$ in a closed hyperbolic $3$--manifold is at least $(ch)^{2h}$.  In \cite{HanRaoWan}, Han, Rao, and Wan prove an analogous lower bound for the number of commensurability classes of closed incompressible surfaces of genus $h$ in any cusped hyperbolic $3$--manifold and then appeal to our type-preserving homomorphisms of $\pi_1(M_8)$ into mapping class groups.

\subsection{Historical notes and other approaches}

\subsubsection{Historical comments}
As far as we are aware, the first mentions of Questions \ref{question.hyperbolic.bundle} and \ref{question.PPA} in the literature occur in the 1990s, though we expect that Question \ref{question.PPA} has been around longer than that. For example, Misha Kapovich \cite{KapovichPrivate}  attributes Question \ref{question.hyperbolic.bundle} to Geoff Mess in 1991, and noted that Question \ref{question.PPA} arose naturally at the same time.

Question \ref{question.hyperbolic.bundle} appears as Question 15 of \cite{Kapovich1998}, published 1998, and was put to W.~J.~Harvey in the 1990s by L. Potyagailo \cite{GonzalezDiezHarvey1999}---see also Question 4.1 of \cite{Reid}.  

Question \ref{question.PPA} was also put to Harvey by Potyagailo \cite{GonzalezDiezHarvey1999} sometime in the nineties, and variants were asked by Mosher \cite{Mosher1997} in 1997 and Kapovich \cite{Kapovich1998} in 1998, who both asked if there are purely pseudo-Anosov groups that are not free.  See also Questions 1.1 of \cite{Reid}, Problem 4.1 of \cite{Mosher2006problems}, and Question 1.3 of \cite{LeiningerReid}.
Kapovich also asked \cite{Kapovich1998} if there are $\delta$--hyperbolic surface-by-surface groups, while Mosher \cite{Mosher1997} and M. Mj \cite{Mitra} asked if there are hyperbolic extensions $1 \to K \to \Gamma \to H \to 1$ where $K$ is nonelementary and $H$ is not free.  

A natural question between Questions \ref{question.big} and \ref{question.hyperbolic.bundle} (asked to us by D.~Fisher) is whether there are surface bundles over surfaces that admit metrics of negative sectional curvature.  
While we could not find this exact question in the literature,  Reznikov's 1993 paper \cite{Reznikov}, particularly his Corollary F.3, makes it clear that this question has also been around for a while.

A variant of Question~\ref{Q:hyperbolic-n} was posed by Reid (see Question~4.10 \cite{Reid}).  The type--preserving assumption here is important, as M. Bridson \cite{Bridson} showed that many hyperbolic $n$--manifold groups embed into mapping class groups, including {\em all} hyperbolic $3$--manifold groups, and, more generally, all virtually special groups.

\subsubsection{Other surface groups} \label{sec:other surface groups}
Representations of surface groups into mapping class groups arise naturally in topology and algebraic topology as monodromies of families of Riemann surfaces. See \cite{Atiyah1969}, \cite{Hirzebruch1969}, \cite{Kodaira1967},
\cite{GonzalezDiezHarvey1999},
\cite{BryanDonagi2002},
and
\cite{SalterTshishiku2020}, for example.
In light of Theorem \ref{theorem.ppAsurfaces}, it is natural to look for other examples of purely pseudo-Anosov surface groups in mapping class groups, and for the reader's benefit we collect here some notable surface subgroups one might look to. 

The \textit{universal} example of a family is the universal curve, which assigns to each point of the moduli space the corresponding algebraic curve. 
This family is a quotient of the Bers fibration over Teichm\"uller space \cite{Bers1973}, which is naturally identified with the Teichm\"uller space of the one punctured surface. Furthermore, the action of $\pi_1(S)$ on the fiber gives a faithful representation to the mapping class group of the once--punctured surface whose image is the Birman kernel \cite{Birman}.
By a theorem of I. Kra \cite{Kra}, an element of $\pi_1(S)$ is pseudo-Anosov under this representation if and only if it is a filling loop, and so these surface groups cannot be purely pseudo-Anosov.

Even more examples arise from the universal curve, for if you take a pseudo-Anosov $g \colon S \to S$, the fundamental group of the mapping torus is the extension group of $\langle g \rangle$, and thus injects into the mapping class group of the once--punctured surface via the Birman Exact Sequence, as described above.
By \cite{MastersZhang2008}, \cite{Baker_Cooper_2000}, \cite{KahnMarkovic2012}, \cite{KahnWright2021}, \cite{CooperFuter}, and \cite{CooperLongReid1994}, these mapping tori contain many immersed incompressible surfaces, and we obtain many surface groups in mapping class groups.
I. Agol has observed \cite{AgolPrivate} that the only reducible elements arising from these representations must lie in the fiber subgroup, where being pseudo-Anosov is equivalent to being filling \cite{Kra}.
As shown in \cite{DowdallKentLeininger2014}, any purely pseudo-Anosov finitely generated subgroup arising in this way will be convex cocompact.  A similar convex cocompactness statement holds when $g$ is reducible \cite{LeiningerRussell}, though all purely pseudo-Anosov subgroups are free in this case.
See \cite{KentLeininger2014} for a discussion of this approach and a geometric criterion to be pseudo-Anosov that serves as an alternative to Kra's \cite{KentLeininger2014} in this setting.  The following question remains open and is quite natural as it would provide examples of convex cocompact surface subgroups of mapping class groups.

\begin{question} Are any of the surface subgroups that arise from the mapping torus of $g$ as above purely pseudo-Anosov?
\end{question}

There are several other potential constructions of surface subgroups worthy of investigation.  We describe three here.  

(1) Agol observed that a small number of moduli spaces can be naturally completed to complex hyperbolic orbifolds, and one may attempt to find a totally geodesic surface that dodges the singularities (see \cite[Section~4.3]{Reid}).

(2) One might hope to make the Kahn--Markovi\'c/Kahn--Wright strategy work in the moduli space (Kahn and Wright's work \cite{KahnWright2021} was motivated by this).

(3) At the 2007 Cornell Topology Festival, W.~Thurston proposed yet another approach to the second author, similar in spirit to the constructions in \cite{Atiyah1969}, \cite{Kodaira1967}, and \cite{GonzalezDiezHarvey1999}.  Thurston suggested finding a ``sufficiently complicated'' multisection of the trivial bundle $S \times S$, which produces a map of the base into the configuration space of several points on $S$, and hence a representation of the fundamental group into the mapping class group.  
While this is not too far from the ideas of the authors and Wright discussed below, the authors could never crystallize Thurston's ideas into a theorem.

\begin{question} Can (1), (2), or (3) be carried out to produce purely pseudo-Anosov closed surface subgroups of mapping class groups?

\end{question}

For completeness, we mention some examples that are known {\em not} to be purely pseudo-Anosov.  Perhaps the most robust construction comes from embeddings of right-angled Artin groups into mapping class groups (considered for example in \cite{CrispParis}, \cite{ClayLeiningerMangahas}, \cite{Koberda}, \cite{Runnels}, and \cite{Seo}).  Many right-angled Artin groups contain numerous surface subgroups  (see \cite{DromsServatius}, \cite{Abrams}, \cite{GordonLongReid}, \cite{CrispWiest}, \cite{CrispSageevSapir}, \cite{KimSam}, and \cite{KimThesis}), but these will {\em never} be purely pseudo-Anosov---see \cite{ClayLeiningerMangahas}.  Another construction inspired from Kleinian groups is given in \cite{LeiningerReid}, and while ``most'' elements are pseudo-Anosov, there is always at least one conjugacy class of non-pseudo-Anosov elements.

\maketitle
\subsection{Sketch of the proof of Theorem \ref{T:type preserving fig 8}}

Wright has posed \cite{Wright.private} the following question, which arose in relation to conversations between him and Kahn.

\begin{question}[Wright \cite{Wright.private}]\label{question.wright} Does there exist a fixed--point-free homeomorphism  $f \colon S \to S$ of a closed hyperbolic surface $S$ with the property that every essential closed curve $\gamma$ on $S$ is homotopically distinct from and fills with its image $f(\gamma)$? 
\end{question}

Wright's motivation for asking this question was his observation that, if there were such a homeomorphism, then we would have a $\pi_1$--injective map \linebreak $x \mapsto \{x, f(x)\}$ of $S$ into the configuration space of two points on $S$, and,   moreover, the image of the fundamental group would be purely pseudo-Anosov in the fundamental group of the configuration space (viewed as a subgroup of the mapping class group of the twice--punctured $S$), thanks to an analogue of Kra's Theorem \cite{Kra} due to Imayoshi, Ito, and Yamamoto \cite{ImayoshiItoYamamoto}.  See Lemma~\ref{lemma.KraTWO.electric.bugaloo} in Section \ref{S:setup} below.

In \cite{Atiyah1969} and \cite{Kodaira1967}, M. Atiyah and K. Kodaira do something similar, though rather than requiring that curves fill with their image, the homeomorphism is taken to be holomorphic---see also G. Gonz\'alez-D\'iez and Harvey \cite{GonzalezDiezHarvey1999}.
Note that a mapping class represented by a holomorphic map is necessarily of finite order, and that, conversely, any finite order mapping class may be realized by a holomorphic map (in fact an isometry of some hyperbolic metric), by J. Nielsen's realization theorem \cite{Nielsen1943}.
The associated bundles admit a complex structure, and since the moduli space of curves is a quasiprojective variety, the map from the base to the moduli space can be taken to be algebraic---by the GAGA principle \cite{Serre}---and thus as in Atiyah's and Kodaira's construction, the bundle forms a complex algebraic family.
On the other hand, D. Futer \cite{Futer2025} has observed that a finite order homeomorphism can never have the second property in Question~\ref{question.wright}.

To illustrate the difficulty of Wright's question, note that T. Aougab, Futer, and S. Taylor \cite{AougabFuterTaylor} have recently shown that the number of fixed points of a pseudo-Anosov homeomorphism is coarsely bounded below by its curve complex translation length, which must be at least three if essential curves fill with their image.

One of our key innovations here is our observation that fixed--point free homeomorphisms of surfaces not only give us representations of the associated surface groups into mapping class groups, but representations of the fundamental groups of the mapping tori---see Corollary \ref{C:extended representation}. 
This implies that if Wright's question has an affirmative answer for some pseudo-Anosov $f$, the resulting purely pseudo-Anosov surface subgroup would have infinite index in its normalizer, forbidding it from being convex cocompact.\footnote{Wright was also aware of this normalizing behavior and its implication for non-convex cocompactness.}

While we are unable to answer Wright's question, we do have available a fixed--point-free homeomorphism of a \textit{noncompact} surface $S$ with the property that every essential loop fills with its image. 
To describe this, first consider
\[
L=
	\left[
   		 \begin{matrix}
   		 2   &   1   \\
   		 1   &   1   \\
   		 \end{matrix}
	\right]
\]
acting on $\RR^2$ by left multiplication on column vectors.  This linear transformation descends to an affine homeomorphism $f \colon T^2 \to T^2$ of the torus $T^2 = \RR^2/\ZZ^2$, viewed as equivalence classes of column vectors.  As above, puncturing at the point $\mathbf{0}$, we write $T^2_\mathbf{0} = T^2 - \{\mathbf{0}\}$ and denote the restriction of $f$ to this punctured surface as
\[
	f_\mathbf{0} = 
	f|_{T^2_{\mathbf{0}}} \colon T^2_{\mathbf{0}} \to T^2_{\mathbf{0}},
\]
which is the monodromy of the fibration of the figure--eight knot complement.  A simple check shows that every essential loop in $T^2_\mathbf{0}$ fills with its $f_\mathbf{0}$--image---see Lemma \ref{lemma.ffillyfilly} in Section \ref{S:setup}.  A Lefschetz fixed point calculation shows that $f_\mathbf{0}$ has no fixed points---see Lemma \ref{lemma.fFPF}, also in Section \ref{S:setup}.

As mentioned above, we obtain a homomorphism from $\pi_1(M_{f_\mathbf{0}}) = \pi_1(M_8)$ into the mapping class group of the twice--punctured fiber.  It is useful to puncture the fiber $T_\mathbf{0}$ at a fixed point $z$ of the square of $f_\mathbf{0}$ and its image $fz$, resulting in the surface $T^2_{\mathbf{0}, z, fz} = T^2 - \{\mathbf{0}, z, fz\}$ and our representation
\[
	\Delta_f \colon \pi_1(M_{f_\mathbf{0}}) \to \Mod(T^2_{\mathbf{0},z,fz}).
\]
It suffices (and is more convenient) to prove that $\Delta_f$ restricted to the index two subgroup, $\pi_1(M_{f_{\mathbf{0}}^2}) < \pi_1(M_{f_\mathbf{0}})$ is type preserving. 

Now suppose $R$ is a reducible mapping class in $\Delta_f(\pi_1(M_{f_\mathbf{0}^2}))$.  
By Wright's observation,  if $R$ is contained in the image of the fiber subgroup $\pi_1(T^2_\mathbf{0})$, then it is necessarily peripheral---see Lemma~\ref{lemma.KraTWO.electric.bugaloo}.  
It therefore suffices to assume that $R$ is not in the fiber subgroup.
The reducing system in $T^2_{\mathbf{0},z,fz}$ must contain a curve bounding a disk in $T^2$, since $f^2$ is Anosov on $T^2$.  
This disk must contain at least two of the points from $\{\mathbf{0},z,fz\}$, and so we may forget one of them and extend $R$ over the twice--punctured torus to a mapping class that is still reducible.

If we forget $z$ or $fz$, we obtain a homomorphism of $\pi_1(M_{f_\mathbf{0}^2})$ into the mapping class group of the twice--punctured torus, and it is not too difficult to see that this is, up to conjugacy, the ``usual'' representation given by the Birman exact sequence \eqref{Eq:embedding SES}.  Rather surprisingly, the same is true if we forget $\mathbf{0}$---see Lemma \ref{L:any 1 of 3}.   The problem of showing that the image of $R$ in the mapping class group of the twice--punctured torus necessarily comes from a peripheral element of $\pi_1(M_{f_\mathbf{0}^2})$ can then be translated into a $3$--manifold topology problem that is easily solved with classical techniques, thus completing the proof of the theorem---see the proof of Theorem~\ref{S:main theorem statement} in Section \ref{Section.forgetting.punctures}.

As the proof of Lemma \ref{L:any 1 of 3} in the case that we forget $\mathbf{0}$ is perhaps the most mysterious point, we briefly describe the idea in this case.  Using an explicit description of $\Delta_f$---see Corollary \ref{C:extended representation} and Section \ref{S:main theorem statement}---we are reduced to proving that the restriction to the fiber subgroup is an isomorphism onto the Birman kernel.  To prove this, given $\gamma$ in $\pi_1(T^2_{\mathbf{0}})$, we view $\Delta_f(\gamma)$ as a $2$--strand braid on $T^2_{\mathbf{0}}$ through $z,fz$.  Forgetting $\mathbf{0}$, we obtain a $2$--strand braid on $T^2$ over $z$ and $fz$, and using the group action of $T^2 \cong S^1 \times S^1$ on itself, we ``straighten'' one strand to produce a $1$--strand braid on $T^2_z$.  The fact that the homomorphism is an isomorphism is then a simple computation.

\begin{remark}
The argument works just as well for any affine homeomorphism of trace $3$ and determinant $1$ or trace $1$ and determinant $-1$.
\end{remark}

\medskip
\noindent \textbf{Acknowledgments.}
The authors have had many fruitful conversations with many people over the last two decades while considering the problems treated here. 
We extend special thanks to
Ian Agol,
Jeff Brock,
Ken Bromberg,
Danny Calegari,
Spencer Dowdall,
Nathan Dunfield,
Matthew Durham,
Cameron Gordon,
Daniel Groves,
Ursula Hamenst\"adt,
Jeremy Kahn,
Darren Long,
Dan Margalit,
Vlad Markovi\'c,
Joe Masters,
Yair Minsky, 
Mahan Mj,
Kasra Rafi,
Alan Reid,
Jacob Russell,
Saul Schleimer,
Alessandro Sisto,
Matthew Stover,
William Thurston,
and
Alex Wright for many illuminating conversations.
Alex Wright, in particular, deserves special thanks for so generously sharing his approach via Question~\ref{question.wright} with us. 
We would also like to thank
Ian Agol,
David Fisher,
David Futer,
Misha Kapovich,
Sang-hyun Kim,
Dan Margalit,
Curt McMullen,
Mahan Mj,
Alan Reid,
Sam Taylor, and
Alex Wright
for comments on an earlier draft of this paper.
We thank the referee for their detailed comments that have significantly improved the exposition.
We thank Cameron Gordon and Alan Reid for teaching us persistence.

\section{Birman's Exact Sequence} \label{S:BES}

Suppose $S$ is a compact orientable surface punctured at a finite set of points.  If $X \subset S$ is any finite subset, we write $S_X$ to denote $S-X$, the surface $S$ punctured at the additional set of points $X$.  Any homeomorphism $h \colon S \to S$ with $h(X) = X$ restricts to a homeomorphism $h_X \colon S_X \to S_X$.  Conversely, any homeomorphism $S_X \to S_X$ that preserves the punctures associated to $X$---the {\em $X$--punctures}---arises in this way.  If $X = \{x\}$,  we  write $S_x = S_X$ and $h_x = h_X$.  We write $\LB h_X \RB$ for the isotopy class of $h_X$.  

Given $S$ and a (possibly empty) finite subset $X \subset S$ as above, we write $\Mod(S_X)$ for the mapping class group of $S_X$, the group of orientation preserving homeomorphisms of $S_X$ up to isotopy, and 
\[
	\Mod(S_X,X) = \{ \LB h_X \RB \ | \ \mathrm{with} \ h(x) = x \ \mathrm{for \ all} \ x \in X \}
\]
for the mapping classes fixing each of the $X$--punctures.
The assignment $h_X \mapsto h$ descends to a homomorphism
\[
	\rho_X \colon \Mod(S_X,X) \to \Mod(S),
\]
for if $h_X$ and $h_X'$ are isotopic, so are $h$ and $h'$.  The homomorphism $\rho_X$ ``forgets'' (or ``fills in'') the $X$ punctures.

If $h_X \colon S_X \to S_X$ represents an element of $\ker(\rho_X)$, then there is an isotopy 
\[
H \colon S \times [0,1] \to S
\]
from $h$ to the identity $\mathds{1}_S$.  We write $h^t = H( \, \cdot \, , t) \colon S \to S$ for this $1$--parameter family of homeomorphisms where $h^0 = h$ and $h^1 = \mathds{1}_S$.  
In this case, $t \mapsto h^t(X)$ is a loop in $\Conf_n(S)$, the configuration space of $n=|X|$ ordered distinct points on $S$.  
Recall that this is defined to be the open subset of the $n$--fold product given by
\[
	\Conf_n(S) = S \times S \times \cdots \times S - \left\{(x_1,\ldots,x_n) \ | \ x_i = x_j \ \mathrm{for \ some} \ i \neq j\right\}.
\]
The loop $t \mapsto h^t(X)$ represents an element of $\PB_n(S) = \pi_1(\Conf_n(S),X)$, the $n$--strand pure braid group on $S$ (here we choose any ordering on $X$, and also use $X$ to denote the $n$--tuple of its points that this ordering defines).  Birman proved that for most surfaces $S$, this braid is uniquely determined by the element of the kernel---see Chapter 4 of \cite{Birman}, \cite{Bers1973}, and Theorem 9.1 of \cite{FarbMargalit}.

\begin{theorem}[Birman Exact Sequence] \label{T:Birman}
If $\chi(S) < 0$ and $X \subset S$ is a set of $n$ points, then there is an exact sequence
\[
	1 \to \PB_n(S) \to \Mod(S_X,X) \stackrel{\rho_X}{\longrightarrow} \Mod(S) \to 1,
\]
where the identification of an element of the kernel of $\rho_X$ with $\PB_n(S)$ is as described above.
\end{theorem}
Given a braid in $\PB_n(S) = \pi_1(\Conf_n(S),X)$ represented by a loop $\gamma$ in $\Conf_n(S)$ based at $X$, we may write $\gamma(t) =(\gamma_1(t),\ldots,\gamma_n(t))$, where $\gamma_1,\ldots,\gamma_n$ are paths in $S$ with
\[
	(\gamma_1(0),\ldots,\gamma_n(0)) = X = (\gamma_1(1),\ldots,\gamma_n(1)).
\]
The Isotopy Extension Theorem \cite[Chapter 8, Theorem 1.3]{Hirsch} implies that there is an isotopy $g^t \colon S \to S$ with $g^0 = \mathds{1}_S$ and $g^t(X) = \gamma(t)$ for all $t$ in $[0,1]$.  We think of $g^t$ as ``point pushing'' $X$ along $\gamma$.  The reversed isotopy $h^t = g^{1-t} \colon S \to S$ is an isotopy from $h^0=g^1 \colon S \to S$ to the identity $h^1=\mathds{1}_S$, so $h^t$ point pushes $X$ along $\overline{\gamma}$, the path with the reversed orientation.  
This identifies the braid $[\gamma]$ with the mapping class in $\Mod(S_X,X)$ obtained from $\mathds{1}_S$ by point pushing {\em backwards} along $\gamma$ and restricting to $S_X$.

\begin{notremark} The notation from the discussion above will be used \linebreak throughout the paper.  Specifically, subscripts on surfaces and maps will denote points being punctured and superscripts on maps will typically denote parameters for isotopies.
\end{notremark}

\bigskip

The pure mapping class group $\PMod(S)$ of $S$ is the subgroup of $\Mod(S)$ consisting of elements represented by homeomorphisms that fix each puncture.  
If $S$ is closed and $Y \subset X \subset S$, then $\PMod(S_X) < \Mod(S_X,Y)$ and we can forget $Y$, defining a (sub) short exact sequence (when $\chi(S_{X-Y}) < 0$)
\[
	1 \to \PB_m(S_{X-Y}) \to \PMod(S_X) \stackrel{\rho_Y}{\longrightarrow} \PMod(S_{X-Y}) \to 1
\]
where $m = |Y|$.

For $z$ in $S$ and $\chi(S) < 0$, another important instance of this exact sequence may be written as
\[
	1 \to \pi_1(S,z) \to \PMod(S_z) \to \PMod(S) \to 1,
\]
since the $1$--strand braid group is just the fundamental group of $S$.

\subsection{Closed braids in mapping tori} \label{S:braids in tori}

Given a homeomorphism $h \colon S \to S$, the mapping torus $M_h$ is the quotient
\[
	M_h = S \times [0,1]/(x,0) \sim (h(x),1).
\]
We write $\frak q \colon S \times [0,1] \to M_h$ for the quotient map.
The fundamental group $\pi_1(M_h)$ is a semidirect product
\[
	\pi_1(M_h) \cong \pi_1(S) \rtimes \ZZ
\]
where the stable letter acts by $h_*$ on $\pi_1(S)$.  
We describe this explicitly when $h$ fixes a basepoint $z$ in $S$, as we make use of this throughout.  
With such an $h$, the induced map $h_* \colon \pi_1(S,z) \to \pi_1(S,z)$ is a well--defined automorphism (not just an outer automorphism).  
We then let $\frak q \colon S \times [0,1] \to M_h$ denote the quotient map and consider the inclusion $\iota \colon S \to M_h$ by $\iota(x) = \frak q(x,\tfrac12)$.  We consider $z$ as a basepoint for both $S$ and $M_h$ by identifying $z$ with $\iota(z)$. We further write $\gamma = \iota \circ \gamma$, and identify $\pi_1(S,z)$ with its image in $\pi_1(M_h,z)$ under $\iota_*$.  
The loop $\tau \colon [0,1] \to M_h$ defined by $\tau(t) = \frak q(z,t+\tfrac12)$ (modulo $1$) then represents the stable letter in the semidirect product, and, for all $\gamma$ in $\pi_1(S,z)$, we have
\[
	\tau \gamma \tau^{-1} = h_*(\gamma).
\]
Puncturing $S$ at $z$, we may then view $M_{h_{\{z\}}} \subset M_h$, where $M_{h_{\{z\}}}$ is obtained from $M_h$ by deleting the image of $\tau$.

If a finite set $X \subset S$ is preserved by a homeomorphism $h \colon S \to S$, then
\[ M_{h_X} = M_h - \mathcal L,\]
where $\mathcal L = \frak{q}(X \times [0,1])$.  

\begin{lemma} \label{L:braids and links 1}
Suppose $h \colon S \to S$ is a homeomorphism fixing a finite set $X \subset S$ with $\LB h_X \RB$ in $\ker(\rho_X)$, and let $h^t$ be the isotopy from $h$ to the identity.  Then 
\[ M_{h_X} \cong S \times S^1 - \mathcal L_0,\] 
where $\mathcal L_0$ is the image of the embedding $X \times S^1 \to S \times S^1$, $(x,t) \mapsto (h^t(x),t)$.\end{lemma}
\begin{proof}
The isotopy $h^t$ defines a homeomorphism $H \colon S \times [0,1] \to S \times [0,1]$ given by $H(y,t) = (h^t(y),t)$.  This descends to a homeomorphism
\[
	\widehat H \colon M_h \to M_{\mathds{1}_S} \cong S \times S^1.
\]
To see this, observe that points $(y,0)$ and $(h(y),1)$ identified in the domain are mapped by $H$ to $H(y,0) = (h(y),0)$ and $H(h(y),1) = (h(y),1)$, which are identified in the range.
The image $\mathcal L_0 = H(\mathcal L) \subset S \times S^1$ of $\calL$ is a link in ``closed braid form'' in the product, being transverse to the fibers $S \times \{\ast\}$ of the product structure.  This allows us to view $M_{h_X} = S \times S^1 - \mathcal L_0$.
\end{proof}

More generally, if $f \colon S \to S$ is \textit{any} homeomorphism, an isotopic homeomorphism can be expressed as $fh \colon S \to S$, where $h \colon S \to S$ is isotopic to the identity.  

\begin{lemma} \label{L:braids and links 2} If $f,h \colon S \to S$ are homeomorphisms fixing a finite set $X$ and $h$ is isotopic to the identity by an isotopy $h^t$, then $H(y,t) =(h^t(y),t)$ defines a homomorphism $M_{fh} \to M_f$, and
\[ M_{(fh)_X} \cong M_f - H(\mathcal L),\]
where $\mathcal L  \subset M_f$ is the image $\mathcal L = \frak{q}(X \times [0,1])$ by the quotient $\frak{q} \colon S \times [0,1] \to M_f$.
\end{lemma}
\begin{proof} This is nearly identical to the proof of Lemma \ref{L:braids and links 1}, so we omit it.
\end{proof}

We now return to the case of a homeomorphism $f \colon S \to S$ fixing a point $z$.   When the mapping class of $f$ has infinite order in $\Mod(S)$, the short exact sequence from the semidirect product embeds into the Birman Exact Sequence
\begin{equation} \label{Eq:embedding SES}
\begin{tikzcd}[row sep=.5cm]
	1 \arrow[r] & \pi_1(S,z) \arrow[r] \arrow[d] & \pi_1(M_f,z) \arrow[r] \arrow[d] & \langle \tau \rangle  \arrow[r] \arrow[d] & 1 \\
	1 \arrow[r] & \pi_1(S,z) \arrow[r] & \Mod(S_z,z) \arrow[r] & \Mod(S) \arrow[r] & 1
\end{tikzcd}
\end{equation}
The vertical homomorphism $\langle \tau \rangle \to \Mod(S)$ sends $\tau$ to the mapping class of $f$ and the image of $\pi_1(M_f,z)$ is precisely the preimage $\rho_z^{-1}(\langle f \rangle)$, with the loop $\tau$ in $M_f$ sent to the mapping class $f_z$.  
The next lemma describes how to think of general elements of $\pi_1(M_f,z)$.  We continue to write $\frak{q} \colon S \times [0,1] \to M_f$ for the quotient map and identify $(S,z)$ as a subset of $(M_f,z)$ via the inclusion $\iota \colon S \to M_f$ given by $\iota(x) = \frak{q}(x,\tfrac12)$.

\begin{lemma} \label{L:identifying loops} Suppose $f,h \colon S \to S$ are homeomorphisms that both fix a point $z$ in $S$.  Further suppose there is an isotopy $h^t$ from $h$ to the identity on $S$, and let $\tau$ and $\beta$ be the loops defined by $\tau(t) = \frak q(z,t+\frac12)$ (modulo 1) and $\beta(t) = h^t(z)$, for $t$ in $[0,1]$.  

Then $t \mapsto \frak q(h^t(z),t)$ is a loop in $M_f$ in the conjugacy class of $[\tau \beta]$ in $\pi_1(M_f,z)$, up to basepoint change isomorphism. Furthermore, the image of this element in $\Mod(S_z,z)$ is $\LB (fh)_z \RB = \LB f_zh_z \RB$ up to conjugacy.
\end{lemma}
\begin{proof} We adjust the isotopy $h^t$ to be constant on initial and terminal intervals, specifically so that $h^t = h$ for $t$ in $\left[0, \tfrac14\right]$ and $h^t = \mathds{1}_S$ for $t$ in $\left[\tfrac12 , 1\right]$ (thus the effect of the isotopy all occurs for $t$ in $\left[\tfrac14,\tfrac12 \right]$).  This does not change the free homotopy class of loops $t \mapsto \frak{q}(h^t(z),t)$ in $M_f$, the homotopy classes of loops $\tau$ and $\beta$, nor the mapping classes $\LB (fh)_z) \RB = \LB f_zh_z \RB$, and so it suffices to prove the lemma under this assumption on $h^t$.  Let $H \colon M_{fh} \to M_f$ be the homeomorphism $H(y,t) = (h^t(y),t)$.  

The loop $\gamma(t) = \frak{q}(h^t(z),t)$ is based at $\frak{q}(z,1)$.  We change the basepoint back to $z = \iota(z) = \frak{q}\left(z,\tfrac12\right)$ using the path $\frak{q}\left(\{z\} \times \left[ \tfrac12, 1 \right] \right)$. This basepoint change sends the homotopy class of $\gamma$ to the homotopy class of a loop based at $z$ that traverses $\frak{q}\left(\{z\} \times \left[ \tfrac12,1 \right]\right)$, then $\frak{q}\left(\{z\} \times \left[ 0,\tfrac14\right] \right)$, and then the loop $\beta$ on $\frak{q}\left( S \times \left[ \tfrac14,\tfrac12 \right] \right)$---see Figure \ref{DDT}.  This loop is clearly in the homotopy class of $\tau \beta$, proving the first part of the lemma.

\def\arr{-{Stealth[length=3mm, width=2mm]}}
\begin{figure}[htb] \label{DDT}
\begin{center}
\begin{tikzpicture}
\node at (0,0) {\includegraphics[width = 8cm]{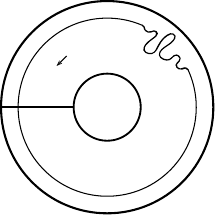}};
\filldraw[white] (-1.8,1.8) circle (.45);
\filldraw (3.33,0) circle (.05);
\draw[thick,dotted] (1.3,0) -- (4,0);
\draw[thick,\arr] (.1,-3.32) -- (-.1,-3.32);
\draw[thick,\arr] (1.96,2.4) -- (2.01,2.5);
\node at (-5,.5) {$\frak{q}(S \times \{0\})$};
\node at (-5,-.5) {$\frak{q}(S \times \{1\})$};
\node at (5.4,0) {$S = \frak{q}(S \times \{\tfrac12\})$};
\node at (3,-.3) {$z$};
\node at (0,-3) {$\tau$};
\node at (1.5,1.5) {$\beta$};
\end{tikzpicture}
\end{center}
	\caption{A representative of the loop $\tau\beta$.}
\end{figure}

This loop $\gamma$ defines an embedding of the circle into $M_f$, and the image is $\mathcal L_0 = H(\mathcal L)$, where $\mathcal L = \frak{q}(\{z\} \times [0,1])$.  By Lemma \ref{L:braids and links 2}, we have
\[ M_{f_zh_z} \cong M_{(fh)_z} \cong M_f - \mathcal L_0.\]
Since $[\tau]$ is sent to $\LB f_z\RB$ and $[\beta]$ to $\LB h_z \RB$, the loop $[\tau \beta]$ is sent to $\LB f_z h_z \RB = \LB (fh)_z \RB$, which proves the ``furthermore'' statement and completes the proof.
\end{proof}

\section{The dancing representation} \label{S:dancing}

We continue to assume that $S$ is a compact orientable surface punctured at a finite set of points.
Let $f$ be any fixed--point-free homeomorphism of $S$. Then there is a map $D \co S \to \Conf_2S$ given by 
\[
	D(x) = (x, f(x)).
\]
This induces a representation 
\[
	D_* \co \pi_1(S,z) \to \pi_1(\Conf_2S, (z, f(z)))=\PB_2(S)
\]
and the inclusion from the Birman Exact Sequence allows us to view this as a homomorphism
\[
	\Delta_f \co \pi_1(S,z) \to \Mod(S_{\{z,f(z)\}},\{z,f(z)\}).
\] 

Given a loop $\beta$ in $S$ based at $z$, there is an isotopy $h^t \colon S \to S$ from $h^0 \colon S \to S$ to the identity $h^1 = \mathds{1}_S$ with $h^t(z) = \beta(t)$ and $h^t(f(z)) = f(\beta(t))$.  More concretely, we can view $h^t$ as pushing $z$ and $f(z)$ along $\beta$ and $f(\beta)$ simultaneously, and then $\Delta_f(\beta) = h^0_{\{z,f(z)\}}$.  The next lemma was Wright's motivation for asking Question~\ref{question.wright}.

\begin{lemma} [A.~Wright] \label{lemma.KraTWO.electric.bugaloo}
Suppose $f \colon S \to S$ is a fixed--point-free homeomorphism.  
If $\beta$ lies in $\pi_1(S,z)$ and $\beta$ and $f(\beta)$ are distinct homotopy classes of curves that fill $S$, then $\Delta_f(\beta)$ is pseudo-Anosov.
\end{lemma}
\begin{proof} The main result of \cite{ImayoshiItoYamamoto} completely describes the Nielsen--Thurston type of a two--strand braid on any orientable finite--type surface, and may be applied in our setting to prove the lemma. 
As the full statement is somewhat technical, we give a self--contained proof for the reader's convenience.

Suppose to the contrary that $\Delta_f(\beta)$ is reducible. Since a pure braid is pure as a mapping class (in the sense of \cite{Ivanov}),
we may assume that a homeomorphism $h$ representing $\Delta_f(\beta)$ fixes the isotopy class of a simple closed curve $\gamma$ in $S_{\{z,f(z)\}}$.  
By Lemma~\ref{L:braids and links 1}, there is a link $\mathcal L_\beta$  in $S \times S^1$ transverse to the $S$--fibers whose complement
${S \times S^1 - \mathcal L_\beta}$ is homeomorphic to $M_h$.  
Projecting $S \times S^1$ onto the first factor, $\mathcal L_\beta$ projects to $\beta \cup f\beta$.  

The fixed curve $\gamma$ defines an essential torus $T \subset S \times S^1 - \mathcal L_\beta$ meeting each $S$--fiber in a simple closed curve.
Assume first that $\gamma$ is homotopically nontrivial in $S$.  Then work of Waldhausen \cite[(2.8) Satz]{Waldhausen1967} implies that, after an isotopy preserving the $S$--fibers, we may assume $T = \gamma \times S^1$.  The isotopy replaces $\mathcal L_\beta$ with an isotopic link $\mathcal L_\beta'$.  
Projecting $\mathcal L_\beta'$ onto the first factor of $S \times S^1$ produces the union of two curves $\beta' \cup \beta''$ homotopic to $\beta \cup f(\beta)$.  In particular, $\beta' \cup \beta''$ is disjoint from the projection of $T$, which is $\gamma$.  However, $\beta \cup f(\beta)$ is assumed to fill, hence so does $\beta' \cup \beta''$, and so the homotopically nontrivial curve $\gamma$ on $S$ must be peripheral. 
Therefore, $\gamma$ bounds a once--punctured disk $B \subset S$, and $T$ bounds $B \times S^1$ in $S \times S^1$.  
Since $T$ is essential in its complement, the link $\calL_\beta'$ must nontrivially intersect $B \times S^1$, and hence at least one of $\beta'$ or $\beta''$ must project into $B$.  
If $\beta'$ projects into $B$, then $\beta'$, and hence $\beta$, is peripheral, which then implies that $f\beta$ is peripheral as well. 
No such pair of curves can fill $S$, and so this is a contradiction.
Similarly, if $\beta''$ projects into $B$, we have that $f\beta$, and hence also $\beta$, is peripheral, again leading to a contradiction.
We conclude that $\gamma$ is homotopically trivial in $S$.

Now, since $\gamma$ is homotopically trivial, $T$ bounds a solid torus $V$ in $S \times S^1$ containing $\mathcal L_\beta$.  Since $\Delta_f(\beta)$ is a pure braid, both components of $\mathcal L_\beta$ are homotopic to a core of $V$.  In particular, they are homotopic to each other.  Projecting this homotopy to $S$ determines a homotopy from $\beta$ to $f\beta$, another contradiction.  This exhausts all possibilities for a reducible $\Delta_f(\beta)$, and we conclude $\Delta_f(\beta)$ is pseudo-Anosov.
\end{proof}

When convenient, we write $\Delta_f(\beta) = (\beta,f\beta)$ when given a loop $\beta$ in $\pi_1(S,z)$, which makes sense as an element of $\PB_2(S) = \pi_1(\Conf_2(S),(z,f(z)))$.

\subsection{The configuration space bundle}

Continue to assume that $f \colon S \to S$ is a fixed point free homeomorphism.  
Letting $\widehat f \colon \Conf_2(S) \to \Conf_2(S)$ denote the homeomorphism $\widehat f(x,y) = (f(x),f(y))$, we construct the mapping torus $M_{\widehat f}$ of $\widehat f$, which is a $\Conf_2(S)$--bundle over the circle.
The embedding $D(x) = (x,f(x))$ above defines an embedding
\[
	D \times \mathds{1}_{[0,1]} \colon S \times [0,1] \to \Conf_2(S) \times [0,1]
\]
that descends to an embedding
\[
	\overline{D} \colon M_f \to M_{\widehat f},
\]
since $D \circ f = \widehat f \circ D$.
Write $\Pi \colon \Conf_2(S) \to S$ for the projection $\Pi(x,y) = x$. 
Since $\Pi \circ \widehat f = f \circ \Pi$, we have that $\Pi \times \mathds{1}_{[0,1]}$ descends to a map
\[
	\overline{\Pi} \colon M_{\widehat f} \to M_f.
\]

\begin{lemma} \label{L:good maps and retractions} The composition $D \circ \Pi \colon \Conf_2(S) \to \Conf_2(S)$ is a retraction onto $D(S)$.  Similarly, $\overline{D} \circ \overline{\Pi} \colon M_{\widehat f} \to M_{\widehat f}$ is a retraction onto the image $\overline{D}(M_f)$.
\end{lemma}
\begin{proof} The first claim is just the observation $D \circ \Pi(x,y) = (x,f(x))$.  The second follows from the first, and the definition of $\overline{D}$ and $\overline{\Pi}$.
\end{proof}
\begin{corollary} The induced map $\overline{D}_* \colon \pi_1(M_f) \to \pi_1(M_{\widehat f})$ is injective.
\end{corollary}

Picking a basepoint $z$, we have the isomorphism $f_* \colon \pi_1(S,z) \to \pi_1(S,f(z))$.  
Picking a path $\delta$ from $z$ to $f(z)$ in $S$, we write $\delta_* \colon \pi_1(S,f(z)) \to \pi_1(S,z)$ for the basepoint change isomorphism given by $\delta_*([\gamma]) = [\delta \gamma \overline{\delta}]$, and denote the composition of these two isomorphisms
\[ f^\delta_* = \delta_* f_* \colon \pi_1(S,z) \to \pi_1(S,z).\]
Viewing $z$ in $S \subset M_f$ as in \S\ref{S:braids in tori}, we have $\pi_1(M_f,z) \cong \pi_1(S,z) \rtimes \langle \sigma \rangle$ where the stable letter $\sigma$ acts as $f^\delta_*$.  The image path $D(\delta) = (\delta,f\delta)$ in $\Conf_2(S)$, together with $\widehat f$, similarly defines an isomorphism on $\PB_2(S) = \pi_1(\Conf_2(S),(z,f(z)))$,
\[
	\widehat f_*^{\, \delta} \colon \PB_2(S) \to \PB_2(S), 
\]
and we can thus write
\[
	\pi_1(M_{\widehat f},(z,f(z))) \cong \PB_2(S) \rtimes \left \langle \overline{D}_*(\sigma) \right\rangle,
\]
so that $\overline{D}_*(\sigma)$ acts as $\widehat f_*^{\, \delta}$, and the isomorphism $\overline{D}_*$ preserves the semidirect product structure, restricting to $D_* \colon \pi_1(S,z) \to \pi_1(\Conf_2(S),(z,f(z))) = \PB_2(S)$ on the normal subgroup.  

\begin{proposition} \label{P:ses embedding} Suppose $f$ is fixed--point-free and $\LB f \RB$ in $\Mod(S)$ has infinite order.  Then the short exact sequence coming from the semidirect product structure on $\pi_1(M_{\widehat f})$ embeds into the Birman Exact Sequence
\[ \label{Eq:embedding SES2}
\begin{tikzcd}[row sep=.5cm]
	1 \arrow[r] &[-.5cm]  \PB_2(S) \arrow[r] \arrow[d] &[-.5cm]  \pi_1(M_{\widehat f},(z,f(z))) \arrow[r] \arrow[d] & \left\langle \overline{D}_*(\sigma) \right\rangle  \arrow[r] \arrow[d] &[-.5cm]  1 \\
	1 \arrow[r] & \PB_2(S) \arrow[r] & \Mod(S_{\{z,f(z)\}},\{z,f(z)\}) \arrow[r,"\rho_{\{z,f(z)\}}"] & \Mod(S) \arrow[r] & 1.
\end{tikzcd}
\]
\end{proposition}
\begin{proof} Let $g \colon S \to S$ be a homeomorphism so that $g(f(z)) = z$ and $g(f^2(z))=f(z)$,  so that $g$ is isotopic to the identity by an isotopy $g^t$ with $g^0 = g$, $g^1 = \mathds{1}_S$, and so that $g^t(f(z)) = \delta(t)$ and $g^t(f^2(z)) = f \delta(t)$.  Then
\[
	(g^tf,g^tf) \colon \Conf_2(S) \to \Conf_2(S),
\]
defines an isotopy from $(gf,gf)$ to $(f,f)$ and $(g^tf(z),g^tf(f(z)) = (\delta(t),f\delta(t))$. Consequently, for any loop $(\alpha,\beta)$ in $\Conf_2(S)$ based at $(z,f(z))$, we have
\[
	(gf\alpha,gf\beta) \simeq \left(\delta (f\alpha)\overline{\delta},(f\delta)(f\beta)(f\overline{\delta})\right)
\]
as loops based at $(z,f(z))$.  See e.g. \cite[Lemma 1.19]{Hatcher}.

We now construct the embedding of short exact sequences.  For this, we start by defining the homomorphism on the stable letter $\overline{D}_*(\sigma)$ in $\pi_1(M_{\widehat f})$, sending it to the element $\LB (gf)_{\{z,f(z)\}} \RB$.  Any element of the kernel of $\rho_{\{z,f(z)\}}$ is represented by $h_{\{z,f(z)\}} \colon S_{\{z,f(z)\}} \to S_{\{z,f(z)\}}$, and we let $h^t$ be an isotopy with $h^0=h$ and $h^1=\mathds{1}_S$, so that $\LB h \RB$ corresponds to the pure braid represented by the loop 
\[
	t \mapsto (\alpha(t),\beta(t))=(h^t(z),h^t(f(z))).
\]
Conjugating $\LB h_{\{z,f(z)\}} \RB$ by $\LB (gf)_{\{z,f(z)\}} \RB$ defines another element of the kernel of $\rho_{\{z,f(z)\}}$ and the associated braid is represented by the loop
\begin{align*}
	t 	& \mapsto (gf) h^t (gf)^{-1}(z,f(z))\\
		& =  gfh^t(z,f(z)) \\
		& =  (gf(\alpha(t)),gf(\beta(t)))\\
		& \simeq  \left(\big(\delta (f\alpha)\overline{\delta}\big)(t),\big((f\delta)(f\beta) (f(\overline{\delta}))\big)(t)\right), 
\end{align*}
as explained above.  It follows that the image of $\overline{D}_*(\sigma)$ conjugates the image of $(\alpha,\beta)$ to the image of the $\overline{D}_*(\sigma)$ conjugate of $(\alpha,\beta)$, and thus we have a well--defined homomorphism from $\pi_1(M_{\widehat f},(z,f(z)))$ to $\Mod(S_{\{z,f(z)\}},\{z,f(z)\})$ that is the ``identity'' on $\PB_2(S)$.  Sending the quotient $\langle \overline{D}_*(\sigma) \rangle$ to $\langle \LB f \RB \rangle$ then completes the embedding of the short exact sequence.
\end{proof}

We write $\Delta_f \colon \pi_1(M_f) \to \Mod(S_{\{z,f(z)\}})$ for the composition of $\overline{D}_*$ with the embedding from the Proposition.
\begin{corollary} \label{C:extended representation} If $f$ is fixed point free and $\LB f \RB$ in $\Mod(S)$ has infinite order, then
\[ \label{Eq:embedding SES3}
\begin{tikzcd}[row sep=.5cm]
	1 \arrow[r] &[-.5cm] \pi_1(S,z) \arrow[r] \arrow[d,"D_*"] &[-.5cm]  \pi_1(M_f,z) \arrow[r] \arrow[d,"\Delta_f"] & \langle \sigma \rangle  \arrow[r] \arrow[d] &[-.5cm] 1 \\
	1 \arrow[r] & \PB_2(S) \arrow[r] & \Mod\left(S_{\{z,f(z)\}},\{z,f(z)\}\right) \arrow[r,"\rho_{\{z,f(z)\}}"] & \Mod(S) \arrow[r] & 1,
\end{tikzcd}
\]
is an embedding of short exact sequences.
\end{corollary}

If $f^2(z) = z$, we can pass to a $2$--fold cover $M_{f^2} \to M_f$, and the issues with basepoints in the proof of Proposition~\ref{P:ses embedding} disappear.  Write $\tau$ for the loop in $M_{f^2}$ based at $z$ in $S \subset M_{f^2}$  representing the stable letter, as in Section \ref{S:braids in tori}.  Then $\tau$ acts like $f^2_*$ on $\pi_1(S,z)$, and restricting $\Delta_f$ to $\pi_1(M_{f^2})$ we get
\[
	\Delta_f \colon \pi_1(M_{f^2},z) \cong \pi_1(S,z) \rtimes \langle \tau \rangle \to \Mod(S_{\{z,f(z)\}},\{z,f(z)\}),
\]
where $\tau$ is sent to $f^2_{\{z,f(z)\}}$.  We note that $\Delta_f$ on $\pi_1(M_f)$ is type--preserving if and only if this restriction is, so it will suffice to work with this restriction.

\section{Figure eight}

We now focus on the figure--eight knot group.

\subsection{A linear map} \label{S:setup}

Let $L\colon \RR^2 \to \RR^2$ be the linear transformation given by
\[
L=
\left[
    \begin{matrix}
    2   &   1   \\
    1   &   1   \\
    \end{matrix}
\right]
\]
and let $f$ be the induced self-map of $T^2 = \RR^2/\ZZ^2$.  The restriction
\[
	f_\mathbf{0} \colon T^2_\mathbf{0} \to T^2_\mathbf{0}
\]
is the monodromy of the fibration of the figure--eight knot complement.

\begin{lemma}\label{lemma.fFPF} The homeomorphism $f$ fixes exactly one point, $\mathbf{0}$, and its square fixes exactly five points, namely
 $\displaystyle{ \mathbf{0}, \, \, z = \colvec{.2\\.4}, \, \, w = fz = \colvec{.8\\.6}, \, \, \colvec{.4\\.8}}$, and $\colvec{.6\\.2}$.
\end{lemma}

\begin{proof}
Since the action of $f_*$ on $H_1(T^2)$ is given by the matrix $L$, which has trace $3$, the Lefschetz number of $f$ is $-1$.  The local Lefschetz number of any fixed point is $-1$ (since the derivative of $f$ is also given by $L$), and consequently, $f$ has exactly one fixed point, which is necessarily $\mathbf{0}$.  The square has Lefschetz number $-5$, and by the same reasoning, there are exactly five fixed points.  A calculation shows that the five points listed are indeed fixed by $f^2$. 
\end{proof}

\begin{lemma}\label{lemma.ffillyfilly} For every essential loop $\beta$ in $T^2_\mathbf{0}$, 
the homotopy classes of $\beta$ and $f_\mathbf{0} \beta$ are distinct and their union is filling.
\end{lemma}
\begin{proof} 
	Let $\beta$ be an essential loop in $T^2_\mathbf{0}$.  Since $f_\mathbf{0}$ is pseudo-Anosov, it cannot fix the homotopy class of an essential curve, and $\beta$ and $f_\mathbf{0} \beta$ are not homotopic.
	
	If $\beta$ is a simple closed curve, then $f_\mathbf{0}(\beta)$ and $\beta$ are a pair of distinct essential simple closed curves on $T^2_\mathbf{0}$, hence their union fills.
	
	If $\beta$ is \textit{not} a simple closed curve, and $\beta$ and $f_\mathbf{0}\beta$ don't fill, then there is an essential simple closed curve $\gamma$ in the complement of $\beta \cup f_\mathbf{0}\beta$. Applying $f_\mathbf{0}^{-1}$, we see that $f_\mathbf{0}^{-1}\gamma$ and $\beta$ are also disjoint. This means that $\beta$ is disjoint from both $\gamma$ and $f_\mathbf{0}^{-1}(\gamma)$. By the simple closed curve case, we know that $\gamma$ and $f_\mathbf{0}^{-1}(\gamma)$ fill, and so $\beta$ must be inessential, a contradiction.
\end{proof}

\subsection{Notation, conventions, and the main theorem} \label{S:main theorem statement}

Let $X = \{\mathbf{0},z,w\}$ be the first three fixed points of $f^2$ listed in Lemma~\ref{lemma.fFPF}.  We continue to use the notation where, for $Z \subset X$, the map $f^2_Z \colon T^2_Z \to T^2_Z$ is the restriction of $f^2$ to $T^2_Z = T^2-Z$.

The homeomorphism $f_\mathbf{0}$ is fixed--point-free and its square $f_\mathbf{0}^2$ fixes $z$  and $w$, by  Lemma~\ref{lemma.fFPF}.  We let $\Gamma = \pi_1(M_{f_\mathbf{0}^2},z)$ and let $K = \pi_1(T^2_\mathbf{0},z)$ be the fiber subgroup, which is free of rank two. 
We also write $\Gamma$ as $\Gamma =  K \rtimes \langle \tau \rangle$ so that $\tau$ in $\Gamma$ is the stable letter as described in Section \ref{S:braids in tori}.  
Corollary \ref{C:extended representation} now gives us a representation
\[
	\Delta_f \colon \Gamma \to \PMod(T^2_X).
\]
(Technically, this should be denoted $\Delta_{f_\mathbf{0}}$, but we have simplified the notation to avoid further clutter.) With this set up we have $\Delta_f(\tau) = f^2_X$, as described after Corollary~\ref{C:extended representation}.
Given an element $\beta$ of $K$, we may write $\Delta_f(\beta) = (\beta,f_\mathbf{0}\beta)$ when convenient, (harmlessly) blurring the distinction between a loop in $\Conf_2(T^2_\mathbf{0})$) and its corresponding mapping class.

By Lemmas \ref{lemma.KraTWO.electric.bugaloo} and \ref{lemma.ffillyfilly}, the representation $\Delta_f$ sends nonperipheral loops in $K$ to pseudo-Anosov elements of $\PMod(T^2_X)$.  
Our main theorem states that this is also the case for all of $\Gamma$, yielding the precise version of Theorem~\ref{T:type preserving fig 8}.

\begin{theorem}\label{theorem.type} The representation $\Delta_f$ is type--preserving. In other words, the mapping class $\Delta_f (\gamma)$ is reducible if and only if $\gamma$ is a peripheral element of $\Gamma$.
\end{theorem}

We will make use of the additive group structure on $T^2 = \RR^2/\ZZ^2$ with identity $\mathbf{0}$.  
Given $x$ in $T^2$, we let $\mu_x \colon T^2 \to T^2$ be the translation $\mu_x(y) = y+x$, whose inverse is
$\mu_x^{-1}(y) = \mu_{-x}(y) = y-x$.   
Observe that for each fixed point $x$ of $f^2$, we have $f^2_x = \mu_x f^2_\mathbf 0 \mu_x^{-1}$, since $f$ is linear. 
In particular, $\mu_x$ determines a canonical homeomorphism
\[
	\widehat \mu_x \colon M_{f^2_\mathbf{0}} \to M_{f^2_x}.
\]
Explicitly, $M_{f^2_x}$ is an open submanifold of $M_{f^2}$ obtained by deleting the image of $\{x\} \times [0,1] \subset T^2 \times [0,1]$ in the quotient $M_{f^2}$, and $\widehat \mu_x$ is the descent of the homeomorphism $\mu_x \times \mathds{1}_{T^2} \colon T^2 \times [0,1] \to T^2 \times [0,1]$ to $M_{f^2}$ restricted to $M_{f^2_\mathbf{0}}$ on the domain and $M_{f^2_x}$ on the range.

\subsection{Forgetting punctures}\label{Section.forgetting.punctures}

For each $x$ in $X = \{\mathbf{0},z,w\}$, we consider the Birman Exact Sequence
\[
    1 \to \pi_1\big(T^2_{X - \{x\}}, x\big) \to \PMod\big(T^2_X\big) \stackrel{\rho_x}{\longrightarrow} \PMod\big(T^2_{X - \{x\}}\big) \to 1,
\]
where $\rho_x$ forgets $x$.

We also consider two element subsets $Y = \{x,y\} \subset X$, and the Birman Exact Sequence
\[
    1 \to \pi_1\big(T^2_y,x\big) \to \PMod\big(T^2_Y \big) \stackrel{\pi_x}{\longrightarrow} \Mod\big(T^2_y \big) \to 1
\]
where we write $\pi_x$ for the homomorphism that forgets $x$ to distinguish it from the map in the previous sequence.  The domain of $\pi_x$ also depends on the choice of two point set $Y$ containing $x$, which we will make clear in context.  Of course, we can interchange the roles of $x$ and $y$.

For $Y = \{x,y\} \subset X$ a two point set as above, there is an associated isomorphism 
\[
	\eta_{x,y} \colon  \pi_1\big(M_{f^2_x},y\big) \to \pi_y^{-1} \langle f^2_x \rangle < \PMod\big(T^2_Y\big)
\]
coming from the embedding of short exact sequences in \eqref{Eq:embedding SES}.

The stable letter in the semidirect product $\pi_1(M_{f^2_x},y)$ is mapped to $f^2_Y$ by this isomorphism.  Since $M_{f^2_x} \cong M_{f^2_\mathbf{0}}$, we can view the domain of $\eta_{x,y}$ as $\Gamma$, when convenient, after choosing an appropriate basepoint--change isomorphism.  
The next lemma tells us that the image depends only on $Y$.

\begin{lemma}
    For $Y = \{x,y\} \subset X$, we have $\pi_y^{-1} \langle f^2_x \rangle = \pi_x^{-1} \langle f^2_y \rangle$. 
\end{lemma}
\begin{proof}
Observe that $\pi_x(f^2_Y) = f^2_y$ and $\pi_y(f^2_Y) = f^2_x$.  Furthermore, the inclusions of $\pi_1(T^2_x,y)$ and $\pi_1(T^2_y,x)$ into $\PMod(T^2_Y)$ via the Birman Exact Sequence have the same image: they are both equal to the kernel of the forgetful map $\PMod(T^2_Y) \to \Mod(T^2)$ since $\Mod(T^2_x) \cong \Mod(T^2)$ with the isomorphism obtained by forgetting the puncture $x$. Therefore, we have
\[
	\pi_y^{-1}\langle f^2_x \rangle = \big\langle \pi_1\big(T^2_x,y \big),f^2_Y \big\rangle = \big\langle \pi_1\big(T^2_y,x\big),f^2_Y \big\rangle = \pi_x^{-1} \langle f^2_y \rangle. \qedhere
\]
\end{proof}
\noindent For any two element subset $Y = \{x,y\} \subset X$, we write $\Gamma_Y = \pi_y^{-1} \langle f^2_x \rangle = \pi_x^{-1}\langle f^2_y\rangle$.  
We also write $K_Y \triangleleft \Gamma_Y$ for the fiber subgroup (which, as the proof shows, is independent of the forgotten point).

We now come to the key lemma.
\begin{lemma} \label{L:any 1 of 3}
    For each $x$ in $X = \{\mathbf{0}, z, w \}$, the composition $\rho_x \circ \Delta_f$ is an isomorphism onto $\Gamma_{X -\{x\}}$, sending $K$ to the fiber subgroup $K_{X-\{x\}}$.
\end{lemma}
We note that although the proof below shows directly that the isomorphism sends $K$ to $K_{X-\{x\}}$ in each case, this actually follows immediately since $\Gamma$ (and hence each $\Gamma_Y$) has a unique homomorphism to $\ZZ$, since $f^2$ is Anosov.
\begin{proof}[Proof of Lemma \ref{L:any 1 of 3}] 

	First consider the case $x = w$ and let $Y = \{\mathbf{0},z\}$ with associated forgetful homomorphism $\pi_z$.  Then $\rho_w \circ \Delta_f(\tau) = \rho_w(f^2_X) = f^2_Y$, and for any $\gamma$ in $\pi_1(T^2_\mathbf{0},z)$, we have
\[
	\rho_w \circ \Delta_f(\gamma) = \rho_w(\gamma,f_\mathbf{0}\gamma) = \gamma, 
\]
viewing $(\gamma,f_\mathbf{0}\gamma)$ as an element of $\pi_1(\Conf(T^2_\mathbf{0}),\{z,w\}) < \PMod(T^2_X)$ and $\gamma$ as an element of $\pi_1(T^2_\mathbf{0},z) < \PMod(T^2_Y)$.  It follows that $\rho_w \circ \Delta_f$ is an isomorphism onto $\Gamma_Y$.  A similar argument holds for $x=z$ (where $Y = \{\mathbf{0},w\}$ with associated forgetful homomorphism $\pi_w$).  The only exception is that the displayed equation becomes
\[
	\rho_z \circ \Delta_f(\gamma) = \rho_z(\gamma,f_\mathbf{0}\gamma) = f_\mathbf{0}\gamma.
\]    

We now consider the case that $x=\mathbf{0}$ and let $Y = \{z,w\}$.  Let us first restrict our attention to the fiber subgroup $K$.

Let $\gamma$ be an element of $K$, and let $h \colon T^2 \to T^2$ be a homeomorphism so that $h_X \colon T^2_X \to T^2_X$ represents $\Delta_f(\gamma)$.  The mapping class $\rho_\mathbf{0} \circ \Delta_f(\gamma)$ is then represented by the homeomorphism 
    \[
        h_Y \colon T^2_Y \to T^2_Y.
    \]

Since $\Delta_f(K)$ is in the kernel of the homomorphism obtained by forgetting \textit{both} $z$ and $w$, the restriction $h_{\mathbf{0}} \colon T^2_\mathbf{0} \to T^2_\mathbf{0}$ is isotopic to the identity on $T^2_\mathbf{0}$.  We let $h^t \colon T^2 \to T^2$ be the extension of that isotopy to an isotopy from $h = h^0$ to the identity $h^1=\mathds{1}_{T^2}$.

\begin{remark} Here and in the equations below, it is helpful to keep in mind that subscripts denote punctured points and superscripts denote parameters for isotopies.
\end{remark}
    
    We define a new isotopy from $h$ to the identity by 
    \[
        \bar h^{\, t} = h^t - h^t(z) + z.
    \]
    For each $t$, the map $\bar h^{\, t}$ is a homeomorphism (namely, the homeomorphism $h^t$ composed with the translation $\mu_{z-h^t(z)}$) and so $\bar h^{\, t}$ does indeed define an isotopy from $\bar h^0$ to $\bar h^1$ on $T^2$.
    Now, since $h^0(z) = z = h^1(z)$, we have
    \begin{align*}
        \bar h^0  &   = h^0 - h^0(z) + z  = h^0
    \end{align*}
and
    \begin{align*}
        \bar h^1  &   = h^1 - h^1(z) + z  = h^1  = \mathds{1}_{T^2}
    \end{align*}
    Also note that, for all $t$, we have
    \begin{align*}
        \bar h^{\, t}(z)  &   = h^t(z) - h^t(z) + z = z.
    \end{align*}
    Since $\bar h^{\, t}$ fixes $z$ for all $t$, the isotopy $\bar h^{\, t}_z \colon T^2_z \to T^2_z$ is well--defined.

    Since $\bar h^{\, t}_z$ is an isotopy from $\bar h^0_z = h^0_z$ to $\bar h^1_z = h^1_z = \mathds{1}_{T^2_z}$, the homeomorphism 
    \[
        h_Y \colon T^2_Y \to T^2_Y
    \]
    representing $\rho_\mathbf{0} \circ\Delta_f(\gamma)$ lies in the kernel of the map 
    \[
        \pi_w \colon \PMod\big(T^2_Y\big) \to \Mod\big(T^2_z\big)
    \]    
    forgetting $w$.
    Moreover, this mapping class corresponds to the element of $\pi_1(T^2_z, w)$ (via the Birman sequence) represented by the loop 
    \begin{align*}
        \bar h^{\, t}_z(w)   &   = h^t(w) - h^t(z) + z     \\
                    &   = f_\mathbf{0}\gamma(t) - \gamma(t) + z.
    \end{align*}

    At this point, we have shown that, for every $\gamma$ in $K$, the mapping class $\rho_\mathbf{0} \circ \Delta_f(\gamma)$ corresponds to the element of $K_Y = \pi_1(T^2_z, w)$ represented by the loop $f_\mathbf{0}\gamma(t) - \gamma(t) + z$.  Now, consider the generators of $K = \pi_1(T^2_{\bf 0}, z)$ given by the loops $A(t) = z + \colvec{t\\0}$ and $B(t) = z + \colvec{0\\t}$.
    For these loops, we have
    \begin{align*}
        \rho_\mathbf{0} \circ \Delta_f(A)   & \approx f_\mathbf{0} \, A(t) - A(t) + z    \\
                            & = f_\mathbf{0} \, \left( z + 
                                        \left[  \begin{matrix}
                                                    1 \\ 0
                                                \end{matrix}
                                        \right] 
                                        t 
                                \right)
                                - \left[  \begin{matrix}
                                                    1 \\ 0
                                                \end{matrix}
                                        \right] 
                                        t 
                                -z
                                +z \\
                            & = \left[
                                    \begin{matrix}
                                        2   &   1   \\
                                        1   &   1   \\
                                    \end{matrix}
                                \right]
                                \left( z + 
                                        \left[  \begin{matrix}
                                                    1 \\ 0
                                                \end{matrix}
                                        \right] 
                                        t 
                                \right)
                                - \left[  \begin{matrix}
                                                    1 \\ 0
                                                \end{matrix}
                                        \right] 
                                        t 
                                \\
                            & = fz + \left[  \begin{matrix}
                                                    1 \\ 1
                                                \end{matrix}
                                        \right] t
                                \\
                            & = w + \left[  \begin{matrix}
                                                    1 \\ 1
                                                \end{matrix}
                                        \right] t
    \end{align*}

    and 
    \begin{align*}
        \rho_\mathbf{0} \circ \Delta_f(B)   & \approx f_\mathbf{0} \, B(t) - B(t) + z    \\
                            & = f_\mathbf{0} \, \left( z + 
                                        \left[  \begin{matrix}
                                                    0 \\ 1
                                                \end{matrix}
                                        \right] 
                                        t 
                                \right)
                                - \left[  \begin{matrix}
                                                    0 \\ 1
                                                \end{matrix}
                                        \right] 
                                        t 
                                -z
                                +z \\
                            & = \left[
                                    \begin{matrix}
                                        2   &   1   \\
                                        1   &   1   \\
                                    \end{matrix}
                                \right]
                                \left( z + 
                                        \left[  \begin{matrix}
                                                    0 \\ 1
                                                \end{matrix}
                                        \right] 
                                        t 
                                \right)
                                - \left[  \begin{matrix}
                                                    0 \\ 1
                                                \end{matrix}
                                        \right] 
                                        t 
                                \\
                            & = fz + \left[  \begin{matrix}
                                                    1 \\ 0
                                                \end{matrix}
                                        \right] t
                                 \\
                            & = w + \left[  \begin{matrix}
                                                    1 \\ 0
                                                \end{matrix}
                                        \right] t.
    \end{align*}
    These two loops form a basis for the free group $K_Y = \pi_1(T^2_z, w)$, and so $\rho_\mathbf{0} \circ \Delta_f|_K$ is a homomorphism taking a basis to a basis and therefore must be an isomorphism onto its image $K_Y$.

    The mapping class $\rho_{\mathbf{0}} \circ \Delta_f(\tau)$ is represented by the homeomorphism $f^2_Y$. 
    Therefore, $\rho_{\mathbf{0}}\circ \Delta_f$ restricts to a homomorphism 
\[
	\Gamma \cong K \rtimes \langle \tau \rangle \to K_Y \rtimes \langle f^2_Y \rangle \cong \Gamma_Y,
\]
taking $K$ isomorphically to $K_Y$ and $\tau$ to $f^2_Y$. 
    This implies that $\rho_{\mathbf{0}}\circ \Delta_f$ is an isomorphism from $\Gamma$ to $\Gamma_Y$, as required.
\end{proof}

\begin{proof}[Proof of the Theorem \ref{theorem.type}] We have already noted that for $\gamma$ in $K$, Lemmas \ref{lemma.KraTWO.electric.bugaloo} and \ref{lemma.ffillyfilly} imply that $\Delta_f(\gamma)$ is reducible if and only if $\gamma$ is peripheral.  Thus, we consider $\tau^m \gamma$ for some $m \neq 0$.  
Then $\Delta_f(\tau^m \gamma)$ is represented by a homeomorphism $f_X^{2m} h_X$ for $h_X$ in the kernel of the homomorphism $\pi_w \rho_z$ that forgets both  $z$ and $w$. 
Suppose $f_X^{2m}h_X$ is reducible. 
After passing to a power (and rewriting using the normal form for semidirect products), we may assume that $f_X^{2m} h_X$ fixes an essential simple closed curve $\alpha \subset T^2_X$. 
(In fact, raising to this power is unnecessary in our setting but this argument suffices.)
 The mapping torus $M_{f_X^{2m}h_X}$ then contains an essential torus $\calT$ meeting each $T^2_X$--fiber transversely in an essential simple closed curve.

We consider $M_{f^{2m}_Xh_X}$ an open submanifold of $M_{f^{2m}h}$.
By Lemma \ref{L:braids and links 2}, there is a $T^2$--fiber-preserving homeomorphism $M_{f^{2m}h} \cong M_{f^{2m}}$ and $M_{f^{2m}_Xh_X}$ is the complement of a link $\mathcal L \subset M_{f^{2m}}$ transverse to the $T^2$ fibers.  Moreover, $\mathcal L$ must have three components since each point of $X$ is fixed by $f^{2m}h$.  The homeomorphism sends $\calT$ to a torus meeting each fiber $T^2$ in a simple closed curve, and the torus defines a homotopy from one such curve to its $f^{2m}$--image.  Since $f^{2m}$ is Anosov, it preserves no homotopy class of essential simple closed curve on $T^2$, and hence each curve of intersection of $\calT$ with the fiber $T^2$ is null homotopic in $T^2$.  In particular, $\calT$ bounds a solid torus $V$ in $M_{f^{2m}}$.

Since $\calT$ is essential in $M_{f^{2m}_Xh_X} = M_{f^{2m}} - \mathcal L$, the solid torus $V$ must contain at least two components of $\mathcal L$, and each component must be a core of $V$.  Let $Y \subset X$ be the set of two points determining this two component link $\mathcal L_0 \subset \mathcal L$ contained in $V$, and let $x$ in $X - Y$ be the third point, whose corresponding component of $\mathcal L$ may or may not lie in $V$.  
Observe that $\rho_x\Delta_f(\tau^m\gamma)$ is represented by $f_Y^{2m}h_Y$ in $\Gamma_Y$, and $M_{f^{2m}_Yh_Y} = M_{f^{2m}} - \mathcal L_0$.  In particular, since $\mathcal L_0 \subset V$, it follows that the torus $\calT$ is still essential in $M_{f^{2m}_Yh_Y}$, and hence $f^{2m}_Yh_Y$ is reducible.

Now write $Y = \{y,u\}$ and note that $\rho_x \Delta_f(\tau^m \gamma)$ in $\Gamma_Y$ is represented by $f^{2m}_Yh_Y$. 
Forgetting $u$, we have that $h_y$ must in fact be isotopic to the identity in $T^2_y$, and we let $h^t_y$ be the isotopy from $h^0_y$ to the identity.  This traces out a loop $\beta(t) = h^t_y(u)$ that represents the element of $\pi_1(T^2_y,u)$ corresponding to mapping class $h_Y$.  By Lemma \ref{L:braids and links 2}, the mapping torus $M_{f^{2m}_Yh_Y}$ is homeomorphic to $M_{f^{2m}_y} - \mathcal K$, where $\mathcal K$ is the knot traced out in $T^2_y \times [0,1]$ by $(h^t_y(u),t)$ and projected to $M_{f^{2m}_y}$.  By Lemma~\ref{L:identifying loops}, the knot $\mathcal K$ is a loop based at $u$ representing $\tau^m \gamma$ which maps to $\rho_x \Delta_f(\tau^m\gamma)$ (up to conjugacy). 
Since $\mathcal K$ is contained in $V$, it follows that this loop is peripheral.  
Since peripheral elements are precisely those whose centralizers are isomorphic to $\ZZ^2$, the isomorphism from $\Gamma \to \Gamma_Y$ maps peripheral elements precisely to the peripheral elements, and hence $\tau^m\gamma$ is peripheral in $\Gamma$, as required.
\end{proof}

\begin{remark} 
We note that our notion of type--preserving here is not the most natural one from a metric perspective.  By Bers's proof \cite{Bers1978} of the Nielsen--Thurston classification theorem, reducible mapping classes are precisely those whose infimal translation length on Teichm\"uller space with respect to the Teichm\"uller metric is not realized as a minimum.  These further split into two subtypes: parabolic (those whose translation length is zero) and pseudo-hyperbolic (those with positive translation length).  A metrically natural notion of type--preserving might require that an element is parabolic if and only if its image is parabolic.
The homomorphism $\Delta_f$ sends every parabolic element to a reducible element, but while those in the fiber subgroup are sent to parabolic elements, those \textit{not} in the fiber subgroup are sent to pseudo-hyperbolic elements, and so $\Delta_f$ does not satisfy this stronger notion of type--preserving.
\end{remark}

\subsection{Surface subgroups}

We have the following consequence of Theorem~\ref{theorem.type}.
\begin{corollary}\label{corollary.surfaces} There are infinitely many commensurability classes of purely \linebreak pseudo-Anosov closed surface subgroups of $\Mod(T^2_X)$.
\end{corollary}
\begin{proof}	The figure--eight knot complement contains infinitely many commensurability classes of totally geodesic closed immersed surfaces \cite{Maclachlan}, and the $\Delta_f$--image of these are purely pseudo-Anosov, by Theorem \ref{theorem.type}.  We claim that infinitely many distinct commensurability classes remain distinct commensurability classes in $\Mod(T^2_X)$.  

Suppose $G_1,G_2 < \Gamma$ are two closed surface subgroups such that $\Delta_f(G_1)$ and $\Delta_f(G_2)$ are conjugate by an element $g$ in $\PMod(T^2_X)$.  Since $K = \pi_1(T^2_{\mathbf{0}},z)$ is free, $G_1$ is not contained in $K$, and so there is an element $\tau^m \gamma$ in $G_1$ with $m >0$ and $\gamma$ in $K$.  Then
\[
	g \Delta_f(\tau^m \gamma) g^{-1} = \big(gf_X^{2m}g^{-1}\big) \big( g\Delta_f(\gamma)g^{-1}\big) 
\]
lies in $\Delta_f(G_2)$.
Applying $\rho_{\{z,w\}}$, we have
\begin{equation}\label{EQ.conjugate}
	\rho_{\{z,w\}}(g \Delta_f(\tau^m \gamma)g^{-1}) = \rho_{\{z,w\}}(g)f^{2m}_{\mathbf{0}} \rho_{\{z,w\}}(g)^{-1}.
\end{equation}
Since $g \Delta_f(\tau^m\gamma) g^{-1}$ is in $\Delta_f(G_2) < \Delta_f(\Gamma)$ and 
$\rho_{\{z,w\}} \colon \Delta_f(\Gamma) \to \Mod(T^2_\mathbf{0})$ has image $\langle f^2_\mathbf{0} \rangle$, the element \eqref{EQ.conjugate} above is in $\langle f^2_{\mathbf 0} \rangle$.  Consequently, $\rho_{\{z,w\}}(g)$ is in the normalizer $N$ of $\langle f^2_{\mathbf{0}} \rangle$.   
Since $f^2_{\mathbf{0}}$ is pseudo-Anosov, $N$ contains $\langle f^2_{\mathbf{0}} \rangle$ with finite index, and hence $g$ is in $\mathcal N = \rho_{\{z,w\}}^{-1}(N)$, which contains $\rho_{\{z,w\}}^{-1}(\langle f^2_\mathbf{0} \rangle)$ with finite index.  
We note that $\mathcal G = \rho_{\{z,w\}}^{-1}(\langle f^2_\mathbf{0} \rangle)$ is the image of $\pi_1(M_{\widehat{f}^{\, 2}},(z,f(z)) < \pi_1(M_{\widehat f},(z,f(z)))$ under the embedding in Proposition~\ref{P:ses embedding}.

Now, suppose there are infinitely many surface subgroups $G_1,G_2,\ldots < \Gamma$ that are pairwise nonconjugate in $\Gamma$, but whose $\Delta_f$--images are conjugate in $\PMod(T^2_X)$.  For each $i \geq 2$, let $g_i$ in $\PMod(T^2_X)$ be an element that conjugates $\Delta_f(G_1)$ to $\Delta_f(G_i)$.  After passing to a subsequence, we can assume that $g_i \mathcal G = g_j \mathcal G$ for all $i,j$, since $g_i$ and $g_j$ lie in   $\mathcal N$ and the index $[\mathcal N : \mathcal G]$ is finite.  But then $g_ig_2^{-1}$ conjugates $\Delta_f(G_2)$ to $\Delta_f(G_i)$, for all $i$, and $g_ig_2^{-1}$ lies in $\mathcal G$ for all $i$.  By Lemma~\ref{L:good maps and retractions}, there is a retraction $r \colon \mathcal G \to \Delta_f(\Gamma)$ induced by the retraction $M_{\widehat f^{\, 2}} \to M_{f^2}$.  It follows that $r(g_ig_2^{-1})$ also conjugates $\Delta_f(G_2)$ to $\Delta_f(G_i)$ for all $i$, but $r(g_ig_2^{-1})= \Delta_f(\gamma_i)$ for some $\gamma_i$ in $\Gamma$.  Since $\Delta_f$ is an isomorphism, $\gamma_i$ conjugates $G_2$ to $G_i$ for all $i$, contradicting the fact that the subgroups $G_i$ were all nonconjugate in $\Gamma$.

We conclude that there are infinitely many $\PMod(T^2_X)$--commensurability \linebreak classes of purely pseudo-Anosov surface subgroups, and since $\PMod(T^2_X)$ has finite index in $\Mod(T^2_X)$, there are also infinitely many $\Mod(T^2_X)$--commensurability classes of such subgroups.
\end{proof}

\section{Mapping class groups of closed surfaces}

Theorem \ref{theorem.ppAsurfaces} follows using a well--known branched covering trick following the ideas of J. Birman and H. Hilden---see \cite{BirmanHilden1971}, \cite{BirmanHilden1973}, \cite{BirmanHilden2017}, and also \cite{MaclachlanHarvey,Winarski,MargalitWinarski}. 
The precise fact we need is the following, which is an immediate consequence of \cite[Lemma 10]{AramayonaLeiningerSouto} (see also \cite[Section 6]{McMullenMukamelWright}).  
Let $\Teich(S)$ denote the Teichm\"uller space of a surface $S$ with its Teichm\"uller metric \cite{GardinerLakic}.  
A homomorphism between torsion--free subgroups of mapping class groups is \textit{type--preserving} if an element is pseudo-Anosov if and only if its image is. 

\begin{proposition}[{\cite[Lemma 10]{AramayonaLeiningerSouto}}]\label{P:branched cover trick}
Suppose $p \colon S \to T^2$ is a covering map branched over $X = \{\mathbf{0},z,w\}$ with local degree at each point of $p^{-1}(X)$ strictly greater than $1$. 
Then there is a finite index subgroup $\Mod_p(T^2_X) < \Mod(T^2_X)$, a type--preserving injective homomorphism $p^* \colon \Mod_p(T^2_X)  \to \Mod(S)$, and a $p^*$--equivariant isometric (and hence totally geodesic) embedding $\Teich(T^2_X) \to \Teich(S)$.  \qed
\end{proposition}

To apply Proposition~\ref{P:branched cover trick} and prove Theorem~\ref{theorem.ppAsurfaces}, we will also need the following.
\begin{proposition} \label{P:all at least 4}
Every closed surface $S$ of genus at least $4$ admits a covering $p \colon S \to T^2$ branched over $X = \{\mathbf{0},z,w\}$ so that the local degree at each point of $p^{-1}(X)$ is strictly greater than $1$.
\end{proposition}
\begin{proof} The {\em branching data} for a branched cover $p \colon S \to T^2$ branched over $X$ is the set $\mathcal P(p) = \{P_\mathbf{0},P_z,P_w\}$, where $P_x$ is the partition of $d$ describing the local degrees over $x$ in $X$.  
That is, for each $x$ in $X$, we have an $r_x$--tuple of nondecreasing positive integers $P_x = (P_x(1),\ldots,P_x(r_x))$, where $r_x = |p^{-1}(x)|$, where $P_x(j)$ is the local degree over $x$ at the $j^{th}$ point of $p^{-1}(x)$, and $P_x(1)+\cdots+P_x(r_x) = d$.  By the Riemann--Hurwitz formula,
\[ \chi(S) = d \chi(T^2) - \left( (d-r_\mathbf{0}) + (d-r_z) + (d-r_w) \right) = r_\mathbf{0} + r_z + r_w - 3d.\]

Husemoller \cite{Husemoller} proved a converse to the Riemann--Hurwitz formula in this setting (see also \cite{EdmondsKulkarniStong}), which says that, given $d \geq 2$ and three partitions $\mathcal P = \{P_\mathbf{0},P_z,P_w\}$ of $d$ such that $r_{\mathbf{0}} + r_z +r_w - 3d$ is a negative even integer $n$, there is a degree $d$ branched cover $p \colon S \to T^2$ with $\mathcal P(p) = \mathcal P$ and $\chi(S) = n$.

To prove the proposition, we need to show that for every even negative integer $n \leq -6$ there is a $d \geq 2$ and branching data $\mathcal P$ satisfying the above conditions {\em and} so that $P_x(j) \geq 2$ for all $x$ in $X$ and $j$ in $\{1,\ldots,r_x\}$.  
This can be done explicitly.  
For $n = -6$, let $d = 3$ and $P_x = (3)$ for all $x$ in $X$.
For $n = -8$, let $d = 4$ and $P_\mathbf{0} = P_z = (4)$ and $P_w = (2,2)$.  For all even $n \leq -10$, there is a positive integer $d \geq 4$ so that $n + 3d$ is in $\{3,4,5\}$.
If $n+3d = 3$, then we can set $P_x = (d)$ for all $x$ in $X$.  If $n+3d = 4$, we set $P_\mathbf{0} = P_z = (d)$ and $P_w = \left( \lfloor \tfrac{d}2\rfloor ,d - \lfloor \tfrac{d}2 \rfloor \right)$.  Finally, if $n+3d=4$, we let $P_\mathbf{0} = (d)$ and $P_z = P_w = \left( \lfloor \tfrac{d}2\rfloor ,d - \lfloor \tfrac{d}2 \rfloor \right)$.
\end{proof}

\subsection{Proof of Theorem \ref{theorem.ppAsurfaces}}

We now give the proof of the main theorem.

\medskip

\noindent
{\bf Theorem~\ref{theorem.ppAsurfaces}} {\em \TheoremInfiniteSurface}

\smallskip

\begin{proof}[Proof of Theorem \ref{theorem.ppAsurfaces}]
For any closed surface $S$ of genus at least $4$, Propositions~\ref{P:all at least 4} and \ref{P:branched cover trick} produce an embedding of finite index subgroups $\Mod_0(T^2_X) < \Mod(T^2_X)$ into $\Mod(S)$, and we identify $\Mod_0(T^2_X)$ with its image inside $\Mod(T^2_X)$. Together with Corollary~\ref{corollary.surfaces}, this produces infinitely many surface subgroups in $\Mod(S)$, and we let $\Omega$ be the set of these subgroups.  We must show that $\Omega$ contains infinitely many $\Mod(S)$--commensurability classes.

Proposition~\ref{P:branched cover trick} produces a totally geodesic equivariant embedding of Teichm\"uller spaces, and we identify $\Teich(T^2_X)$ with its image in $\Teich(S)$.  Let $\Lambda$ be the stabilizer of $\Teich(T^2_X)$, which contains $\Mod_0(T^2_X)$ with finite index.  Consequently, every $\Lambda$--commensurability class of subgroups in $\Omega$ is a finite union of $\Mod_0(T^2_X)$--commensurability classes.
As every $\Mod(S)$--commensurability class in $\Omega$ is a union of $\Lambda$--commensurability classes, it suffices to prove that this union is finite.  

Fix any $G$ in $\Omega$.  It suffices to prove that if $\{G_k\}_{k=0}^\infty \subset \Omega$ is any infinite collection of subgroups that are all $\Mod(S)$--commensurable with $G = G_0$, then there is a pair of distinct subgroups $G_i$ and $G_j$ that are $\Lambda$--commensurable.  For all $k$, let $h_k$ in $\Mod(S)$ be an element that conjugates a finite index subgroup of $G_k$ to a finite index subgroup of $G$.  

Since every element of $G_k$ is pseudo-Anosov, every element of $G_k$ has a Teichm\"uller geodesic axis contained in $\Teich(T^2_X)$. 
Furthermore, for any finite index subgroup $H_k$ of $G_k$, this same axis is the axis for some pseudo-Anosov in $H_k$. 
Applying $h_k$ to these axes produces axes of elements of $G$, which also lie in $\Teich(T^2_X)$.  
It follows that, for all $k$, the subspace
\[ 
	\Theta_k = h_k(\Teich(T^2_X)) \cap \Teich(T^2_X) 
\]
is an infinite diameter totally geodesic subspace that is invariant under $G$.  
If, for some $k > 0$, we have $h_k(\Teich(T^2_X)) = \Teich(T^2_X)$, then $h_k$ lies in $\Lambda$, which implies that $G_0$ and $G_k$ are commensurable, and we are done.  We therefore assume as we may that all of the $\Theta_k$ are proper totally geodesic subspaces.

Now let $k > 0$.
Since $\Teich(T^2_X)$ is $3$--dimensional over $\CC$, each $\Theta_k$ has dimension $1$ or $2$.  
If a $\Theta_k$ had dimension $1$  it would be a Teichm\"uller disk, and the stabilizer of such a disk can never contain a closed surface subgroup.
We conclude that each $\Theta_k$ has dimension $2$.  
Now, the stabilizers of $\Theta_k$ and $\Theta_\ell$ have a closed surface subgroup in common, and so the stabilizer of $\Theta_k \cap \Theta_\ell$ contains a closed surface subgroup as well. 
Since $\Theta_k$ and $\Theta_\ell$ are totally geodesic, so is $\Theta_k \cap \Theta_\ell$.
If this intersection had dimension one, it would be a Teichm\"uller disk, which is again forbidden by the presence of a surface group in its stabilizer. 
We conclude that the intersection $\Theta_k \cap \Theta_\ell$ has dimension two, and hence that $\Theta_k = \Theta_\ell$.
We may thus rename the $\Theta_k$ as one set $\Theta$, and we do so.

Projecting to the moduli space $\mathcal M(S)$ of $S$, the quotient $\Teich(T^2_X)/\Lambda$ properly immerses as a totally geodesic algebraic subvariety (see e.g.~\cite{Wright2020}).  
Passing to a manifold cover of $\mathcal M(S)$, we see that, at any point of self-intersection of $\Teich(T^2_X)/\Lambda$ in the image of $\Theta$, the self-intersection locus is locally given by finitely many intersecting $3$--dimensional totally geodesic subspaces.  
Therefore $\{h_k(\Teich(T^2_X))\}_{k=1}^\infty$ consists of only finitely many distinct subspaces.  
In particular, there are distinct $k, \ell \geq 1$ so that $h_k(\Teich(T^2_X)) = h_\ell(\Teich(T^2_X))$, and consequently $h_k^{-1}h_\ell(\Teich(T^2_X) ) = \Teich(T^2_X)$.  
But then $h_k^{-1}h_\ell$ lies in $\Lambda$, and  $h_k^{-1}h_\ell$ commensurates $G_k$ to $G_\ell$.  That is, $G_k$ and $G_\ell$ are in the same $\Lambda$--commensurability class, as required.
\end{proof}

\begin{remark} With a little more care, it is not difficult to see that the number of $\Mod_0(T^2_X)$ commensurability classes in any $\Mod(S)$--commensurability class is bounded in a manner depending only on the branched covering $S \to T^2_X$ and the choice of lifting subgroup $\Mod_0(T^2_X)$.  For example, rather than assuming $\{G_k\}$ was an infinite set in the proof, we could have just assumed it contained more than the number of intersecting $3$--dimensional subspaces near the self-intersection point.
\end{remark}

\subsection{Sketch of the proof of Theorem \ref{theorem.counting}}

Closer examination of the proofs of Corollary \ref{corollary.surfaces} and Theorem \ref{theorem.ppAsurfaces} establishes Theorem \ref{theorem.counting}.

\medskip

\noindent
{\bf Theorem~\ref{theorem.counting}}
{\em \TheoremCounting}

\smallskip

\begin{proof}[Sketch of proof.]
Let $\calS_\Gamma$ be the set of conjugacy classes of convex cocompact surface subgroups of $\Gamma$ and, for any surface $\mathfrak S$, let $\calS_{\Mod(\mathfrak S)}$ be the set of conjugacy classes of purely pseudo-Anosov subgroups of $\Mod(\mathfrak S)$.

As in the proof of Corollary \ref{corollary.surfaces}, 
let $\mathcal G = \rho_{\{z,w\}}^{-1}(\langle f^2_\mathbf{0} \rangle)$ and 
let $\mathcal N = \rho_{\{z,w\}}^{-1}(N)$, where $N$ is the normalizer of $\langle f^2_\mathbf{0} \rangle$  in $\Mod(T^2_\mathbf{0})$.
The indices   ${[\Mod(T_X^2) : \PMod(T_X^2)]}$ and $[\mathcal N : \mathcal G]$ are bounded by a universal constant, and tracing through the proof of Corollary \ref{corollary.surfaces} with this in mind reveals that there is a universal constant $m$ such that the the natural map $\calS_\Gamma \to \calS_{\Mod(T^2_X)}$ induced by the representation $\Delta_f$ is $m$-to-$1$. We conclude that the number of commensurability classes of purely pseudo-Anosov surface groups in $\Mod(T^2_X)$ of genus at most $h$ is comparable to the number of commensurability classes of such convex cocompact surface groups in $\pi_1(M_8)$.

Now let $S$ be the branched cover constructed in the proof of Theorem \ref{theorem.ppAsurfaces}, let $\Xi$ be a maximal finite index subgroup of $\Gamma$ such that $\Delta_f(\Xi)$ lifts to a subgroup $\Xi_S$ of $\Mod(S)$, and note that the index of $\Xi$ in $\Gamma$ is bounded by a constant depending only on $\chi(S)$.
As noted after the proof of Theorem \ref{theorem.ppAsurfaces}, the number of $\Xi_S$--commensurability classes of purely pseudo-Anosov surface groups in $\Xi_S$ of bounded genus is comparable to the number of such $\Mod(S)$--commensurability classes, where the constants of comparison depend only on $\chi(S)$.

It follows that the number of $\Mod(S)$--commensurability classes of purely 
\linebreak 
pseudo-Anosov surface groups of genus at most $h$ is comparable to the number of $\pi_1(M_8)$--commensurability classes of convex cocompact surface groups of genus at most $h$.
The proof is completed by noting that the number of commensurability classes of cocompact \textit{fuchsian} subgroups of $\pi_1(M_8)$ of genus at most $h$ is comparable to a linear function of $h$, by the discussion at the end of \cite{Vulakh}.\footnote{We thank M. Stover for bringing \cite{Vulakh} to our attention.}
\end{proof}

\bibliographystyle{plain}
\bibliography{AtoroidalBIB}

{
\footnotesize

\noindent 
\textsc{Department of Mathematics, University of Wisconsin, Madison, WI}
\newline \noindent  \texttt{kent@math.wisc.edu}   

\bigskip
\noindent
\textsc{Department of Mathematics, Rice University, Houston, TX}
\newline \noindent \texttt{cjl12@rice.edu}

\bigskip
\noindent
\textsc{Autumn E. Kent and Christopher J. Leininger}

}

\end{document}

%% file: commands.tex
\newcommand{\HH}{\mathbb H}
\newcommand{\RR}{\mathbb R}

\newcommand{\ZZ}{\mathbb Z}

\newcommand{\CC}{\mathbb C}

\newcommand{\calL}{\mathcal L}

\newcommand{\calG}{\mathcal G}

\newcommand{\calS}{\mathcal S}
\newcommand{\calT}{\mathcal T}

\newcommand{\co}{\colon\thinspace}

\newcommand{\PSL}{\mathrm{PSL}}

\newcommand{\Mod}{\mathrm{Mod}}
\newcommand{\PMod}{\mathrm{PMod}}

\newcommand{\PB}{\mathrm{P}\mathcal{B}}

\newcommand{\Teich}{\mathcal{T}}

\newcommand{\Conf}{\mathrm{Conf}}

